\documentclass[11pt]{amsart}

\usepackage{subfiles}

\usepackage{tabu}
\usepackage[margin=1in]{geometry}
\usepackage[dvipsnames]{xcolor}   		             		
\usepackage{placeins}
\usepackage{graphicx}			
\usepackage{amssymb,amsfonts}
\usepackage{mathrsfs}
\usepackage{amsthm}
\usepackage{amsmath}
\usepackage{stmaryrd}
\usepackage{tikz}
\usepackage{tikz-cd}
\usepackage{accents}
\usepackage{upgreek}
\usepackage{enumerate}
\usepackage{bm}
\usepackage{mathtools}
\usepackage[all]{xy}
\usepackage{caption}

\usepackage{url}
\usepackage{float}
\usepackage{todonotes, derivative}
\usepackage{bbm}

\usepackage[full]{textcomp}
\usepackage[cal=cm]{mathalfa}
\usepackage{xparse}
\usepackage{comment}
\usepackage{cite}
\usepackage{xfrac}
\usepackage[normalem]{ulem}

\tikzset{
  commutative diagrams/.cd, 
  arrow style=tikz, 
  diagrams={>=stealth}
}

\theoremstyle{definition}

\makeatletter
\def\@tocline#1#2#3#4#5#6#7{\relax
  \ifnum #1>\c@tocdepth % then omit
  \else
    \par \addpenalty\@secpenalty\addvspace{#2}%
    \begingroup \hyphenpenalty\@M
    \@ifempty{#4}{%
      \@tempdima\csname r@tocindent\number#1\endcsname\relax
    }{%
      \@tempdima#4\relax
    }%
    \parindent\z@ \leftskip#3\relax \advance\leftskip\@tempdima\relax
    \rightskip\@pnumwidth plus4em \parfillskip-\@pnumwidth
    #5\leavevmode\hskip-\@tempdima
      \ifcase #1
       \or\or \hskip 1em \or \hskip 2em \else \hskip 3em \fi%
      #6\nobreak\relax
    \dotfill\hbox to\@pnumwidth{\@tocpagenum{#7}}\par
    \nobreak
    \endgroup
  \fi}
\makeatother

\usetikzlibrary{calc}
\usetikzlibrary{fadings}
\usetikzlibrary{decorations.pathmorphing}
\usetikzlibrary{decorations.pathreplacing}

\newcounter{marginnote}
\DeclareMathAlphabet{\mathpzc}{OT1}{pzc}{m}{it}

\usepackage[backref=page]{hyperref}
\hypersetup{
  colorlinks   = true,          %Colors links instead of ugly boxes
  urlcolor     = blue,          %Color for external hyperlinks
  linkcolor    = purple,          %Color of internal links
  citecolor   = blue             %Color of citations
}

\theoremstyle{definition}
\newtheorem{theorem}{Theorem}[subsection]

\newtheorem{corollary}[theorem]{Corollary}
\newtheorem{lemma}[theorem]{Lemma}
\newtheorem{proposition}[theorem]{Proposition}
\newtheorem{remark}[theorem]{Remark}

\newtheorem*{runningexample*}{Running example}

\newtheorem*{aside*}{Aside}

\newtheorem{definition}[theorem]{Definition}
\newtheorem{example}[theorem]{Example}

\newtheorem{proposition-definition}[theorem]{Proposition-Definition}

\newtheorem{maintheorem}{Theorem}

\DeclareMathOperator{\Hilb}{Hilb}

\newcommand{\bcd}{\begin{center}\begin{tikzcd}}
\newcommand{\ecd}{\end{tikzcd}\end{center}}

\newcommand{\cI}{\mathcal{I}}

\newcommand{\cX}{\mathcal{X}}

\newcommand{\CC}{\mathbb{C}}

\newcommand{\PP}{\mathbb{P}}

\newcommand{\RR}{\mathbb{R}}

\newcommand{\AAA}{\mathbb{A}}
\newcommand{\GG}{\mathbb{G}}

\newcommand{\frM}{\mathfrak{M}}

\newcommand{\frX}{\mathfrak{X}}

\newcommand{\Spec}{\operatorname{Spec}}

\NewDocumentCommand{\compatibilitydatum}{m m m m m m O{} O{} O{}}{
\begin{equation*} \begin{tikzcd}[ampersand replacement=\&]
  \: \arrow{r} \& {#1} \arrow{r} \arrow{d}{#7} \& {#2} \arrow{r} \arrow{d}{#8} \& {#3} \arrow{r}{[1]} \arrow{d}{#9} \& \: \\
  \: \arrow{r} \& {#4} \arrow{r} \& {#5} \arrow{r} \& {#6} \arrow{r} \& \:
\end{tikzcd} \end{equation*}}

\NewDocumentCommand{\commutingsquare}{m m m m o O{} O{} O{} O{}}{
\begin{equation}\begin{tikzcd}[ampersand replacement=\&] \label{#5}
  #1 \arrow{r}{#6} \arrow{d}{#7} \& #2 \arrow{d}{#8} \\
  #3 \arrow{r}{#9} \& #4
\end{tikzcd}\IfValueTF{#5}{\label{#5}}{} \end{equation}}

\NewDocumentCommand{\cartesiansquare}{m m m m O{} O{} O{} O{}}{
\begin{equation*}\begin{tikzcd}[ampersand replacement=\&]
  #1 \arrow{r}{#5} \arrow{d}{#6} \arrow[dr, phantom, "\square"] \& #2 \arrow{d}{#7} \\
  #3 \arrow{r}{#8} \& #4
\end{tikzcd} \end{equation*}}

\NewDocumentCommand{\cartesiansquarelabel}{m m m m m O{} O{} O{} O{}}{
\begin{tikzcd}[ampersand replacement=\&]
  #1 \arrow{r}{#6} \arrow{d}{#7} \arrow[dr, phantom, "\square"] \& #2 \arrow{d}{#8} \\
  #3 \arrow{r}{#9} \& #4
\end{tikzcd}\IfValueTF{#5}{\label{#5}}{}
}

\NewDocumentCommand{\triangleofspaces}{m m m O{} O{} O{}}{
\begin{tikzcd} [ampersand replacement=\&]
#1 \arrow{r}{#4} \arrow[bend right]{rr}{#5} \& #2 \arrow{r}{#6} \& #3
\end{tikzcd}}

\makeatletter
\newsavebox{\@brx}
\newcommand{\llangle}[1][]{\savebox{\@brx}{\(\m@th{#1\langle}\)}%
  \mathopen{\copy\@brx\kern-0.5\wd\@brx\usebox{\@brx}}}
\newcommand{\rrangle}[1][]{\savebox{\@brx}{\(\m@th{#1\rangle}\)}%
  \mathclose{\copy\@brx\kern-0.5\wd\@brx\usebox{\@brx}}}
\makeatother

\begin{document}
 
\title{From logarithmic Hilbert schemes to degenerations of hyperk\"ahler varieties}
\author[Shafi]{Qaasim Shafi}
\email{mshafi@mathi.uni-heidelberg.de}

\author[Tschanz]{Calla Tschanz}
\email{callatschanz@gmail.com}

\begin{abstract}
    We construct the first examples of good type III degenerations of hyperkähler varieties in dimension greater than 2. These are presented as moduli of 0-dimensional subschemes on expansions of a degeneration of K3 surfaces. We prove projectivity for our expanded degenerations and compute the dual complexes of the special fibre for two specific degenerations of hyperkähler fourfolds. Moreover, we explain the correspondence between geometric strata of the special fibre and simplices in its dual complex.
\end{abstract}

\maketitle
\vspace{-3em}
\tableofcontents

\vspace{-3em}

\section*{Introduction}

By the Bogomolov Decomposition Theorem \cite{Beauville}, compact Kähler manifolds with trivial canonical bundle decompose, up to finite étale cover, into a product of complex tori, Calabi–Yau manifolds, and irreducible holomorphic symplectic (hyperkähler) manifolds. As such, hyperkähler varieties form one of the fundamental building blocks of K-trivial geometry and play a central role in the study of higher-dimensional algebraic and differential geometry. So far, only four deformation types of hyperkähler manifolds are known, the most well studied of which is the Hilbert scheme of points on a K3 surface.

In this paper, we construct and analyse the first explicit example of a good type III degeneration of hyperkähler varieties (see Definition \ref{defn : semistable}) in dimension greater than two. Our construction is a moduli space parametrising zero dimensional subschemes of length $m$ on certain expansions of a type III degeneration of K3 surfaces. For $m=2$, we give an explicit description of the dual complex of the special fibre of our degenerations and explain the correspondence between the geometric strata and and the simplices of the dual complex.

Degenerations of K-trivial varieties lie at the heart of a range of phenomena across geometry. In particular, type III degenerations, also known as maximally unipotent degenerations, play a crucial role in various conjectures involving mirror symmetry, Hodge theory, and moduli theory. For Calabi–Yau and hyperkähler varieties, mirror symmetry predicts \cite{kontsevich2001homological,kontsevich2006affine} that the base of a Lagrangian torus fibration is homeomorphic to the dual complex of a maximally degenerate fibre, and that this complex carries a natural integral affine structure. The perverse filtration of this Lagrangian fibration has been proven to correspond to the weight filtration of such a type III degeneration, through the hyperkähler P = W conjecture \cite{harder2021pw}. In the case of K3 surfaces, understanding the geometry of such degenerations and the associated integral affine structure on the dual complex has played a key role in results concerning the compactification of the moduli space of K3 surfaces \cite{alexeev2023compact}. Being able to study explicit examples of degenerations of hyperkählers as well as their associated dual complexes brings us closer to understanding compactifications of moduli spaces of hyperkählers. Enumerative geometry of hyperkähler varieties, and in particular of the Hilbert scheme of points on a K3 surface have a rich structure and yet there are fundamental questions that remain open \cite{oberdieck2024holomorphic}. Understanding type III degenerations offers an avenue forward using (reduced) degeneration formula techniques \cite{maulik2025logarithmic,blomme2025correlated}.

\subsection{Results}

Let $X \rightarrow C$ be a type III degeneration of K3 surfaces. We focus on the degenerations appearing in Section \ref{Section : K3 degenerations} but the strategy works more generally (see Remark \ref{Remark : general degen}). The relative Hilbert scheme $\mathrm{Hilb}^m(X/C)$ has general fibre the Hilbert scheme of points on a K3 surface, which of course is smooth, but its central fibre $\mathrm{Hilb}^m(X_0)$ can be very singular. This makes it difficult to work with and, in particular, it is not a minimal degeneration. The construction replaces the central fibre in such a way that we obtain a semistable degeneration of the Hilbert scheme of points on a K3 surface. This was achieved in the local situation in \cite{CT} by constructing certain expansions of $X_0$ which allowed families of subschemes in the general fibres to have limits supported away from the singular locus of the central fibre. In this paper, we show how to move from the local situation to an arbitrary degeneration. To illustrate the possibilities, we show that, for the degeneration in Section \ref{Section : Cube degeneration}, globalising is automatic, whereas for the degeneration in Section \ref{Section : Quartic degeneration} a more intricate construction is needed to glue together our local pictures. Through tropical arguments, we show that the resulting expanded degeneration is still projective. The output in either case is a family over a stack of expansions $\frX \rightarrow \mathfrak{C}$. Our main results are the following:

\begin{maintheorem}
    The stack of stable, length $m$ zero dimensional subschemes $\frM_{\mathrm{LW}}^m$ in $\frX$ over $C$, constructed in Section \ref{Section : Globalising}, is a Deligne--Mumford stack, proper over $C$. Moreover, the construction gives a good semistable degeneration of irreducible holomorphic symplectic varieties with general fibre $\mathrm{Hilb}^m(\mathrm{K3)}$. 
\end{maintheorem}

\begin{maintheorem}
     For $m=2$, the dual complex $\Pi_0$ of the special fibre of $\frM_{\mathrm{LW}}^m$ is a 4-dimensional delta-complex. Its $k$-dimensional strata correspond to components of base codimension $5-k$ in the expanded degeneration construction and we give a complete explicit description of this delta complex.
\end{maintheorem}

For the \emph{quartic} degeneration of K3 surfaces (see Section \ref{Section : Quartic degeneration}), the resulting delta-complex $\Pi_Q$ is made up of 10 vertices, 45 edges,  110 triangles, 120 tetrahedrons and 48 4-simplices (See Theorem \ref{count quartic}). For the \emph{cube} degeneration of K3 surfaces (see Section \ref{Section : Cube degeneration}) the resulting delta-complex $\Pi_C$ is made up of 21 vertices, 120 edges, 420 triangles, 480 tetrahedrons and 192 4-simplices (see Theorem \ref{count cube}).

\subsection{Future directions}
The results presented here are an important stepping stone in both the theory of hyperkählers and of logarithmic Hilbert schemes. As the first explicit example of a good type III degeneration of hyperkähler varieties outside of K3 surfaces, it opens up new possibilities in the study of degenerations and moduli space compactifications. In future work, we hope to study the integral affine structures of the dual complexes described in Section \ref{section:dual complex}. One of our aims is to generalise the well-known results on K3 surfaces which state that the integral affine structure on the dual complex of a maximally degenerate K3 surface has 24 singular points (possibly colliding).

One of the interesting results that we discovered during this work is the set of numerical values found in Theorem \ref{count quartic} and the fact that they match those of Bagchi and Datta \cite{BD} for the unique simplicial triangulation of $\CC\PP^2$ with 10 vertices. While it is expected that the Euler characteristics for their construction and ours would be the same, the fact that all numbers should match is far from obvious. Indeed, in dimension greater than two, triangulations of complex projective space with the same number of vertices can have very different numbers of $k$-simplices. Our hypothesis for this phenomenon is that this occurs because of the choice of expanded degeneration we have made. Effectively, this choice gives rise to the smallest possible stack of expansions. This could lead to finding more minimal simplicial triangulations of $\CC\PP^n$ for $n>2$. It also situates this construction as particularly meaningful among all logarithmic Hilbert schemes of points. In a sense, it is minimal among all such constructions.

\subsection{Organisation}
Section \ref{section: background} sets up the necessary background. We briefly recall the expanded degeneration construction of \cite{CT}; define terms in hyperk\"ahler geometry and the correct notions of minimality for our situation; and describe the two specific families of K3 surfaces we will later work with.
Section \ref{Section : Globalising} is where we describe how the local arguments of \cite{CT} can be globalised to work on families of K3 surfaces. We discuss how to glue our local models together in a compatible way; a proof that the expanded degenerations obtained are still projective; and that the constructions are semistable stacks.
Finally, in Section \ref{section:dual complex}, we give an in-depth description of the dual complexes of our degenerations for $m=2$.

\subsection*{Acknowledgements}
Q.S. was supported by the starting grant ’Correspondences in enumerative geometry: Hilbert schemes, K3 surfaces and modular forms’, No 101041491 of the European Research Council. C.T. acknowledges support by the DFG through projects number 530132094 and 550535392.

\section{Background}\label{section: background}

In this section, we review the expanded degeneration construction of \cite{CT}, which produces good degenerations of Hilbert schemes of points on semistable families of surfaces. This construction forms the foundation for our later analysis of dual complexes arising from degenerations of hyperkähler varieties. We then recall the necessary background on hyperkähler manifolds and their minimal degenerations. Finally, we present two explicit type III degenerations of K3 surfaces whose associated Hilbert scheme of points degenerations—constructed via \cite{CT}—will serve as our central examples.

\subsection{Expanded degenerations of surfaces}

We begin by recalling the main features of the expanded degeneration construction from \cite{CT}, which generalizes the approach of Li–Wu \cite{LW} in the setting of Hilbert schemes of points on surfaces. While the construction is originally local, it can be globalized to handle degenerations of K3 surfaces, as we will see in later sections.

\subsubsection{Semistable families of surfaces}

Let $\pi \colon X \to \mathbb{A}^1$ be a flat, projective, semistable family of surfaces over an algebraically closed field $k$ of characteristic zero. Étale locally, this family is modeled by the standard singularity
\[
\Spec k[x, y, z, t]/(xyz - t) \longrightarrow \Spec k[t],
\]
where the special fibre $X_0 = \pi^{-1}(0)$ consists of three coordinate planes meeting transversely:
\[
Y_1 = \{x = 0\}, \quad Y_2 = \{y = 0\}, \quad Y_3 = \{z = 0\}.
\]

\subsubsection{Expanded degenerations}\label{Section : expanded degenerations}

Fix an integer $n \geq 0$, and let $\mathbb{A}^{n+1}$ have coordinates $(t_1, \ldots, t_{n+1})$, referred to as the \emph{basis directions}. For Hilbert schemes of $m$ points, we consider $n=2m$. Define a map
\[
\mathbb{A}^{n+1} \longrightarrow \mathbb{A}^1, \quad (t_1, \ldots, t_{n+1}) \mapsto t_1 \cdots t_{n+1},
\]
and form the fibre product
\[
X \times_{\mathbb{A}^1} \mathbb{A}^{n+1}.
\]
This space is singular along loci where the $Y_i$ intersect and two or more $t_i$ vanish. This allows us to make \emph{small blow-ups}, i.e.\ we blow-up Weil divisors which results in $\PP^1$-bundles appearing over singular curves in the fibres of of $X \times_{\mathbb{A}^1} \mathbb{A}^{n+1}\to \mathbb{A}^{n+1}$. The resulting family, which we call an \emph{expanded degeneration}, is denoted by $X[n]\to C[n]\coloneqq \mathbb{A}^{n+1}$.

More precisely, in the étale local model, $X[n]$ is realized as a closed subscheme of 
\[
X \times_{\mathbb{A}^1} \mathbb{A}^{n+1} \times (\mathbb{P}^1)^n
\]
cut out by the equations
\begin{align*}
    x_0^{(1)} t_1 &= x x_1^{(1)}, \\
    y_0^{(1)} t_{n+1} &= y y_1^{(1)}, \\
    y_1^{(k-1)} y_0^{(k)} t_{n+2-k} &= y_0^{(k-1)} y_1^{(k)} \quad \text{for } 2 \leq k \leq n, \\
    y_0^{(n)} x z &= y_1^{(n)} t_1, \\
    x_0^{(k)} y_0^{(n+1-k)} z &= x_1^{(k)} y_1^{(n+1-k)},
\end{align*}
where $(x_0^{(k)} : x_1^{(k)})$, $(y_0^{(k)} : y_1^{(k)})$ are homogeneous coordinates on the newly introduced $\mathbb{P}^1$ factors from the blow-up process. Note that we have not resolved all singularities of the fibre product here. We have, however, created enough modifications of $X_0$ to choose new limits for every family of length $m$ 0-dimensional subschemes of $X\to \AAA^1$ which have smooth support.

\begin{definition}
The exceptional components introduced in the blow-up process (seen as $\PP^1$-bundles over reducible curves in the singular fibres of $X[n]\to \AAA^1$) are called \emph{$\Delta$-components}. We distinguish two types:
\begin{itemize}
    \item The component $\Delta_1^{(k)}$ corresponds to the $k$-th blow-up in the $x$-direction, associated to coordinates $(x_0^{(k)} : x_1^{(k)})$.
    \item The component $\Delta_2^{(k)}$ corresponds to the $k$-th blow-up in the $y$-direction, associated to coordinates $(y_0^{(k)} : y_1^{(k)})$.
\end{itemize}
\end{definition}
We observe, e.g.\ by studying the equations above, that when $x=y=0$ and $z\neq 0$ the $\Delta_1^{k}$ and $\Delta_2^{n+1-k}$ coincide. In particular, the construction of \cite{CT} is not symmetric. This is a crucial point which allows for this construction to give rise to the \emph{smallest possible stack of expansions} and for the resulting degeneration of Hilbert schemes of points to be a semistable stack.

\begin{definition}[Base codimension]
A fibre of the morphism $X[n] \to \AAA^{n+1}$ is said to have \emph{base codimension} $k$ if exactly $k$ of the $t_i$ vanish in this fibre. Similarly, a length $m$ 0-dimensional subscheme with support in this fibre can be said to have base codimension $k$. This terminology extends also to quotients of $X[n]$.
\end{definition}

\begin{definition}\label{speed}
     We will say that two irreducible components of a length $m$ 0-dimensional subscheme in $X_0$ \emph{fall into the singular locus at the same speed} if their support lies in the singular locus of $X_0$ and their degree of vanishing in $x$, $y$ or $z$ is the same.
\end{definition}

\subsubsection{Torus action and moduli of points}\label{Section : torus action}

The family $X[n] \to \mathbb{A}^{n+1}$ carries a natural action of the torus $\mathbb{G}_m^{n}$, which lifts to the relative Hilbert scheme $\Hilb^m(X[n]/\mathbb{A}^{n+1})$. For $(\tau_1,\dots,\tau_{n})\in \GG_m^{n}$, the action is the following:
\begin{align*}
    (x_0^{(k)} : x_1^{(k)}) &\mapsto (\tau_k x_0^{(k)} : x_1^{(k)})\\
    (y_0^{(n+1-k)} : y_1^{(n+1-k)}) &\mapsto (y_0^{(n+1-k)} : \tau_k y_1^{(n+1-k)}) \\
    (t_1,t_2,\dots,t_n, t_{n+1}) &\mapsto (\tau_1t_1, \tau_1^{-1}\tau_2t_2,\dots, \tau_{n-1}^{-1}\tau_n t_n, \tau_{n}^{-1}t_{n+1}).
\end{align*}

Define the \emph{stack of expansions}
\[
\mathfrak{C}\coloneqq \lim_{\to} [C[n]/\sim],
\]
where the direct limit is over all $n$ and $\sim$ denotes the $\GG_m^{n}$-action as well as some additional isomorphisms (see \cite{CT} for details). These equivalences lift naturally to the space $X[n]$ and we define similarly a \emph{family over the stack of expansions} $\mathfrak{X} \to \mathfrak{C}$. A natural extension of \emph{Li–Wu stability} (LW stability) is developed for zero-dimensional subschemes on $\mathfrak{X}$.

\begin{definition}\label{Def: LW stab}
    Let $Z$ be a length $m$ 0-dimensional subscheme of a fibre of $X[n] \to \mathbb{A}^{n+1}$ or $\mathfrak{X} \to \mathfrak{C}$. The subscheme $Z$ is said to be \emph{Li-Wu (LW) stable} if it has finite automorphisms and is supported in the smooth locus of the fibre. We denote the locus of Li-Wu stable points in $\Hilb^m(X[n]/\mathbb{A}^{n+1})$ by $H^m_{[2m],\mathrm{LW}}$.
\end{definition}

The main result of \cite{CT} is that the associated moduli stack $\mathfrak{M}^m_{\mathrm{LW}}$ of LW-stable zero-dimensional subschemes of length $m$ on $\mathfrak{X}$ is Deligne–Mumford and proper over $C$.

In the situation of \cite{MR} the naive moduli problem is not separated and an additional stability condition called Donaldson--Thomas stability is required. One interesting feature of the construction here is that no auxiliary stability condition is needed to ensure properness. This stems from the fact that the combinatorics of the blow-ups in the construction of $X[n]$ are rigidly controlled. 

\begin{definition}
    The support of a length $m$ 0-dimensional subscheme in a fibre of either $X[n]\to C[n]$ or $\mathfrak{X}\to \mathfrak{C}$ will be referred to as a \emph{configuration of points}. Notably, the configurations of points determine the \emph{combinatorial type} of a length $m$ 0-dimensional subscheme. More precisely, if the points lie in the interior of the same 2-dimensional components (before taking the torus quotient), they are of the same type.
\end{definition}

\subsection{Hyperkähler manifolds and degenerations} 

\subsubsection{Hyperkähler varieties}

\begin{definition}
    A smooth projective variety $X$ over $\CC$ is called irreducible holomorphic symplectic (IHS) or hyperkähler if the space of holomorphic 2-forms $H^0(X,\Omega^2_X)$ is spanned by an everywhere nondegenerate form $\omega$ and $X$ does not admit any non-trivial finite \'etale covering (or its analytification is simply connected).
\end{definition}

It follows from the condition on $H^0(X,\Omega^2_X)$ that $X$ is necessarily even dimensional and has trivial canonical bundle. In the complex analytic category there are four known classes of such objects. The first two were constructed by Beauville \cite{Beauville}: 

\begin{enumerate}
    \item The Hilbert scheme, $\mathrm{Hilb}^n(\mathrm{K3})$, of $n$ points on a K3 surface
    \item The generalised Kummer variety $K^{n+1}(A)$ on a 2-dimensional abelian variety $A$
\end{enumerate}

All known examples of hyperkähler manifolds have been shown to be a deformation of one of these two cases, except for two families of examples constructed by O’Grady~\cite{OG1,OG2}.

\subsubsection{Degenerations of hyperkähler varieties}

\begin{definition}\label{defn : semistable}
    A good degeneration of irreducible holomorphic symplectic varieties is a proper, flat morphism $\pi : \cX \rightarrow C$ of relative dimension $2n$ such that 
    \begin{enumerate}
        \item $\pi$ is semistable, i.e. the total space $\cX$ is smooth and the special fibre $\cX_0 = \pi^{-1}(0)$ is a reduced simple normal crossings divisor.
        \item There exists a relative logarithmic $2$-form $\omega_{\pi} \in H^0(\cX,\Omega_{\cX/C}^2(\log \cX_0))$ such that $\wedge^n \omega_{\pi}$ is nowhere vanishing.
    \end{enumerate}
\end{definition}

\begin{definition}
    Let $\frX$ be a Deligne--Mumford stack which is flat and locally of finite type over $C$. We say that $\frX \rightarrow C$ is semistable if there exists an \'etale atlas $\cX \rightarrow \frX$ such that the composition $\cX \rightarrow C$ is semistable.
\end{definition}

\begin{remark}
   The construction of \cite{CT} gives rise to a semistable degeneration because we work with stacks. In general, the coarse moduli space will have quotient singularities. Consequently, when working with schemes, in the literature, such degenerations are often allowed to have dlt singularities. 
\end{remark}

\begin{definition}\label{DEf dual complex}
  Let $E$ be smooth variety (or Deligne--Mumford stack) together with a reduced simple normal crossings divisor $E_0 = \cup_{a \in A} E_a$. The dual complex $\Delta(E_0)$ is the unique $\Delta$-complex with the following properties:
\begin{itemize}
    \item The $d$-dimensional simplices correspond bijectively to the connected components $E_B^i$ of $E_B = \cap_{b \in B}E_b$, as $B$ runs through the subsets $\emptyset \neq B \subset A$ with $|B| = d + 1$.
    \item Let $B$ and $B'$ be two non-empty subsets of $A$. Then $\mathsf{Simp}(E_B^i)$ is a face of $\mathsf{Simp}(E_{B'}^{i'})$ if and only if $E_{B'}^{i'} \subset E_B^i$.
\end{itemize}
\end{definition}

\begin{definition}
    Let $X \rightarrow C$ be a projective degeneration of hyperkähler manifolds (including the K3 case). Let $\nu \in \{1, 2, 3\}$ be the nilpotency index for the associated monodromy operator $N$ on $H^2(X_t)$ (i.e. $N = \log T_u$, where $T_u$ is the unipotent part of the monodromy $T = T_sT_u$). We say that the degeneration is of type I, II, or III respectively if $\nu = 1, 2, 3$ respectively.
\end{definition}

Moreover, the following theorem of Koll\'ar, Laza, Sacc\'a and Voisin \cite{kollar2018remarks} tells us that given a dlt minimal degeneration, the type of the degeneration is dictated by the dimension of the dual complex.

\begin{theorem}
    Let $X \rightarrow C$ be a minimal dlt degeneration of $2n$-dimensional hyperkähler manifolds. Let $\Delta$ denote the dual complex of the central fibre (and
$|\Delta|$ its topological realisation). Then
\begin{enumerate}
\item $\dim |\Delta|$ is $0, n$ or $2n$ if and only if the type of the degeneration is I, II, or III respectively, i.e. $\dim |\Delta| = (\nu - 1)n$, where $\nu$ is as above.
\item If the degeneration is of type III, then $|\Delta|$ is a simply connected closed pseudo-manifold (see \cite{nicaise2016essential} for details), which is a rational homology $\CC \PP^n$.
.
\end{enumerate}
\end{theorem}

\subsection{Two type III degenerations of K3 surfaces}\label{Section : K3 degenerations} 

K3 surfaces give $2$-dimensional hyperkähler varieties. Here we give two specific examples of type III degenerations of K3 surfaces, for which we compute the dual complex for the corresponding Hilbert square degenerations.

\subsubsection{Quartic degeneration}\label{Section : Quartic degeneration}
Let $f_4(x_0,x_1,x_2,x_3)$ be a general homogeneous quartic polynomial. The vanishing locus inside $\PP^3$ is a smooth quartic K3 surface (of degree $4$). Consider $$\cX' = Z(t\cdot f_4 + x_0 x_1 x_2 x_3) \subset \PP^3 \times \AAA^1_t \rightarrow \AAA^1_t.$$
The general fibre is a quartic K3 surface and the fibre over $0$ is the union of the four coordinate hyperplanes in $\PP^3$. The total space $\cX'$ is not smooth, there are precisely $24$ singularities, $4$ along each of the pairwise intersections of the coordinate hyperplanes. These singularities locally look like the singularity at the origin $Z(ab - cd) \subset \AAA^4$. There are two resolutions, either of these choices produces a semistable degeneration $$\cX \rightarrow \AAA^1_t$$ where the fibre over $0$ still consists of four components glued in the same manner, but now these components are blow-ups of $\PP^2$ in up to $4$ points on each coordinate line.
\subsubsection{Cube degeneration}\label{Section : Cube degeneration} Let $f_{2,2,2}(x_0,x_1,y_0,y_1,z_0,z_1)$ be a general trihomogeneous polynomial of degree $(2,2,2)$. The vanishing locus inside $\PP^1 \times \PP^1 \times \PP^1$ is a K3 surface of degree $6$. Consider $$\cX' = Z(t\cdot f_{2,2,2} + x_0 x_1 y_0 y_1 z_0 z_1) \subset \PP^1 \times \PP^1 \times \PP^1 \times \AAA^1_t \rightarrow \AAA^1_t.$$

As before the general fibre is a K3 surface and in this case the fibre over $0$ consists of $6$ surfaces, each a coordinate $\PP^1 \times \PP^1$ inside $\PP^1 \times \PP^1 \times \PP^1$. Once again, the total space is not smooth, there are 24 singularities, two along each $\PP^1$ which is the intersection of two of the components of the special fibre. These can be resolved similarly to the quartic degeneration, with the same choices in the resolution to produce a semistable degeneration $$\cX \rightarrow \AAA^1_t$$
where the fibre over $0$ still consists of $6$ components glued in the same manner, but now these components are the blow-ups of $\PP^1 \times \PP^1$ in up to $2$ points on each of the $4$ coordinate lines.

\section{Global construction}\label{Section : Globalising}
\subsection{Tropicalisations}
Here, we recall the basic dictionary between our geometric and tropical pictures, in so far as these concepts are used in the following sections. For a more in-depth understanding of the role of logarithmic and tropical geometry in this story, see \cite{CT} and \cite{MR}.

Definition \ref{DEf dual complex} formally defined a dual complex. We will explain what this looks like in the situations that interest us. In the local set-up of \cite{CT}, where we study three planes meeting transversely, the dual complex is given by a triangle. We recall Figure \ref{geom and trop special fibre} from \cite{CT}. For those acquainted with logarithmic geometry, the cone over the triangle corresponds to the tropicalisation of $X$ with respect to the divisorial logarithmic structure given by the divisor $X_0$ in $X$. For the purposes of this work we will only need to work with this dual complex triangle (or slice of the tropicalisation at a given height).
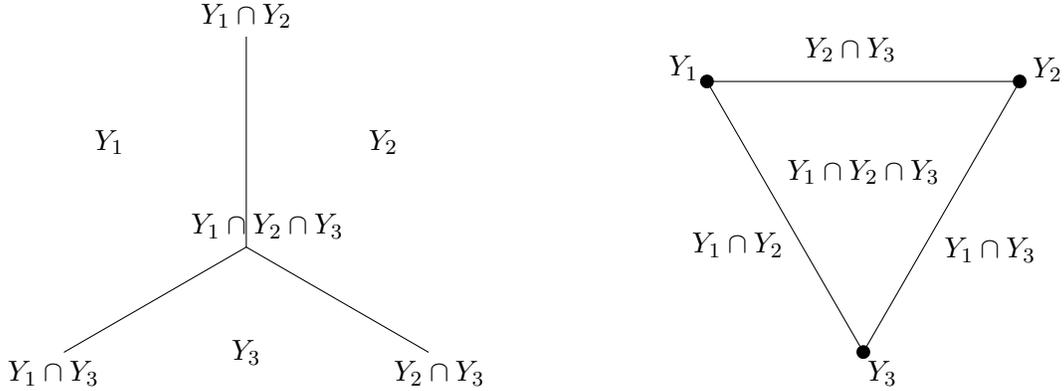
\begin{figure} 
    \begin{center}   
    \begin{tikzpicture}[scale=1.4]
        \draw   
        (0,0) -- (0,2)       
        (-1.732, -1) -- (0,0)
        
        (1.732, -1) -- (0,0);
        \draw (-1.3,1) node[anchor=center]{$Y_1$};
        \draw (1.3,1) node[anchor=center]{$Y_2$};
        \draw (0,-1) node[anchor=center]{$Y_3$};
        \draw (0,2.2) node[anchor=center]{$Y_1\cap Y_2$};
        \draw (-1.832, -1.2) node[anchor=center]{$Y_1\cap Y_3$};
        \draw (1.832, -1.2) node[anchor=center]{$Y_2\cap Y_3$};
        \draw (0.2, 0.2) node[anchor=center]{$Y_1\cap Y_2 \cap Y_3$};
    \end{tikzpicture}
    \hspace{2cm}
    \begin{tikzpicture}[scale=1.2,rotate=60]
        \draw   

        (-1.732, -1) -- (0,2)       
        (-1.732, -1) -- (1.732, -1)
        
        (1.732, -1) -- (0,2);
        \draw (-1.4,0.8) node[anchor=center]{$Y_1\cap Y_2$};
        \draw (1.1,0.8) node[anchor=center]{$Y_2\cap Y_3$};
        \draw (-0.05,-1.65) node[anchor=center]{$Y_1\cap Y_3$};
        \filldraw[black] (0,2) circle (2pt) ;
        \filldraw[black] (-1.732, -1) circle (2pt) ;
        \filldraw[black] (1.732, -1) circle (2pt) ;
        \draw (0,2.3) node[anchor=center]{$Y_1$};
        \draw (-1.832, -1.3) node[anchor=center]{$Y_3$};
        \draw (2, -1.2) node[anchor=center]{$Y_2$};
        \draw (0, 0) node[anchor=center]{$Y_1\cap Y_2 \cap Y_3$};
\end{tikzpicture}
    \end{center}
    \caption{Geometric and tropical pictures of the special fibre $X_0$.}
    \label{geom and trop special fibre}
\end{figure}
We would further like to recall that subdivisions of the tropicalisation or dual complex correspond geometrically to making birational modifications to $X$ in its special fibre. In particular, the blow-ups made in \cite{CT} are iterations of those described in Figure \ref{local subdivision}.
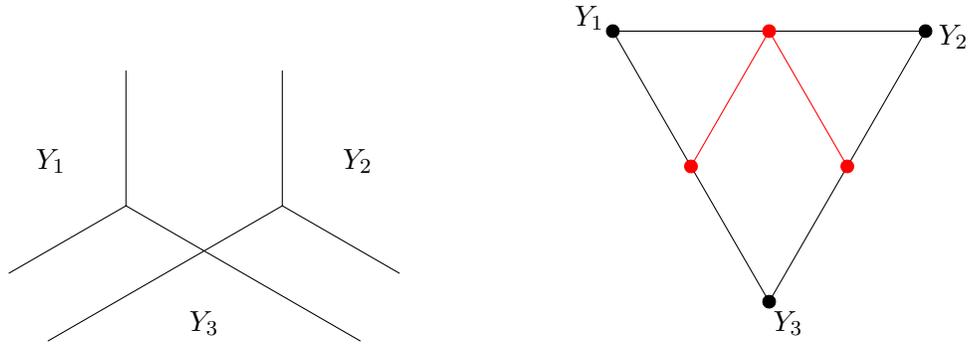
\begin{figure} 
    \begin{center}   
    \begin{tikzpicture}[scale=1.2]
        \draw   (-0.866,0.5) -- (0,0)
        (0,0) -- (0.866,0.5)       
        (-0.866,0.5) -- (-0.866,2)
        (0.866,0.5) -- (0.866,2)
        (-1.732, -1) -- (0,0)
        (-2.165,-0.249) -- (-0.866,0.5)
        
        (1.732, -1) -- (0,0)
        
        (2.165,-0.249) -- (0.866,0.5);
        %\draw (-1.516,-0.375) node[anchor=center]{$\Delta_1^{(1)}$};
        %\draw (1.516,-0.375) node[anchor=center]{$\Delta_2^{(2)}$};
        %\draw (0,1.25) node[anchor=center]{$\Delta_1^{(1)} = \Delta_2^{(2)}$};
        \draw (-1.7,1) node[anchor=center]{$Y_1$};
        \draw (1.7,1) node[anchor=center]{$Y_2$};
        \draw (0,-0.8) node[anchor=center]{$Y_3$};
        
    \end{tikzpicture}
    \hspace{2cm}
    \begin{tikzpicture}[scale=1.2,rotate=60]
\draw   

        (-1.732, -1) -- (0,2)       
        (-1.732, -1) -- (1.732, -1)
        
        (1.732, -1) -- (0,2)
        
        ;
        \draw[red] (-0.866,0.5) -- (0.866,0.5);
        \draw[red] (0,-1) -- (0.866,0.5);
        %\draw (-1.3,0.5) node[anchor=center, color = red]{$\Delta_1^{(1)}$};
        %\draw (1.35,0.7) node[anchor=center, color = red]{$\Delta_1^{(1)} = \Delta_2^{(2)}$};
        %\draw (0,-1.5) node[anchor=center, color = red]{$\Delta_2^{(2)}$};
        \filldraw[red] (-0.866,0.5) circle (2pt) ;
        \filldraw[red] (0.866,0.5) circle (2pt) ;
        \filldraw[black] (0,2) circle (2pt) ;
        \filldraw[black] (-1.732, -1) circle (2pt) ;
        \filldraw[black] (1.732, -1) circle (2pt) ;
        \filldraw[red] (0, -1) circle (2pt) ;
        \draw (0,2.3) node[anchor=center]{$Y_1$};
        \draw (-1.832, -1.3) node[anchor=center]{$Y_3$};
        \draw (1.832, -1.3) node[anchor=center]{$Y_2$};
\end{tikzpicture}
    \end{center}
    \caption{Geometric and tropical picture at $t_1=t_2=0$ in $X[1]$.}
    \label{local subdivision}
\end{figure}

The dual complex of the special fibre $X_0$ in the degenerations of degree 6 or 4 K3 surfaces is given by a triangulation of a sphere. In the quartic case, which we study more closely, this will be the unique triangulation of a sphere with four vertices; we may view it as a tetrahedron.

For the purposes of describing the subdivisions of this tetrahedron (modifications of $X_0$) more easily, we draw it flattened out as in Figure \ref{flattened picture}.
\begin{figure}
    \centering
    \begin{tikzpicture}
        \draw
        %123
        (-1.732, 2) -- (0,0)       
        (-1.732, 2) -- (1.732, 2)
        (0,0) -- (1.732, 2)
        %124
        (-1.732,2) -- (0,4)
        (1.732,2) -- (0,4)
        %134
        (0,0) -- (0,-2)
        (-1.732,2) -- (0,-2)
        (1.732,2) -- (0,-2)
        ;
        \draw (-2,2) node[anchor=center]{$Y_1$};
        \draw (2,2) node[anchor=center]{$Y_2$};
        \draw (0,4.3) node[anchor=center]{$Y_4$};
        \draw (0,0.4) node[anchor=center]{$Y_3$};
        \draw (0,-2.3) node[anchor=center]{$Y_4$};
     \end{tikzpicture}
    \caption{Flattened picture of the triangulated sphere with 4 vertices.}
    \label{flattened picture}
\end{figure}
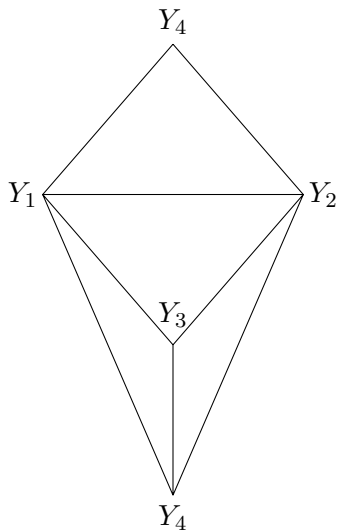
\subsection{Gluing local blow-ups.}\label{Section : Gluing blow-ups}
The subdivision or blow-up depicted in \ref{local subdivision} is an étale local modification around a transverse triple intersection of planes. Here we show how this can be globalised to a surface whose dual complex is the triangulation of a sphere. What this means is that there is a way to subdivide the entire dual complex such that each of the 4 or 6 original triangles carry a subdivision of the type given in Figure \ref{local subdivision}, that these subdivisions glue and that the modification thus defined is projective.

\subsubsection{Cube case.}\label{Section : Cube}
Firstly, we note that the globalisation for the cube degeneration of Section \ref{Section : Cube} is straightforward. Indeed, in this case, it is possible to label all vertices of the dual complex by the numbers $1$, $2$ or $3$, such that each triangle has exactly one of each vertex. We denote by $Y_{(i)}$ the union of geometric components corresponding to vertices labelled $i$.
Then carrying out the construction of Section \ref{Section : expanded degenerations}, replacing $Y_i$ by $Y_{(i)}$, extends the local construction globally.

\subsubsection{Quartic case.}\label{Section : Quartic}
Clearly in the quartic case, where the dual complex can be visualised as a tetrahedron, the above reasoning breaks down and a more subtle argument is required. Indeed, if we tried to follow the above recipe, we could label any three vertices 1, 2 and 3, and the fourth vertex would be left over. Our local construction can still be extended to this case, but we can no longer rely simply on blowing up entire irreducible components of $X_0$. We will start by explaining here how to construct $X[1]$ globally, and $X[n]$ will follow by induction.

\begin{remark}
    We describe here one possible choice of global subdivision which extends the local construction, but several other choices are possible. These different choices would result in different birational models of our degenerate Hilbert scheme of points.
\end{remark}

We take the fibre product $X\times_{\AAA^1}\AAA^2$ as before, with the usual notation $(t_1,t_2)\in \AAA^2$. We describe several local blow-ups, defined around each corner of the fibre $X_0$ where 3 surfaces meet transversely.
\begin{definition}\label{4 BUs}
    These are the four local modifications that we propose. We will see later that they glue together to form a global blow-up.
    \begin{itemize}
        \item In the complement of $Y_4$, we blow up $Y_1$ along the vanishing of $t_1$, followed by the strict transform of $Y_2$ along the vanishing of $t_2$ as before. In local coordinates, the ideals of the blow-ups are $(x,t_1)$ and $(y,t_2)$.
        \item In the complement of $Y_3$, we also blow up $Y_1$ along the vanishing of $t_1$, followed by the strict transform of $Y_2$ along the vanishing of $t_2$ as before. The blow-ups ideals are the same as above.
        \item In the complement of $Y_2$, we blow up $Y_1$ along the vanishing of $t_1$, followed by the strict transform of $Y_3$ along the vanishing of $t_2$. The ideals of the blow-ups are $(x,t_1)$ and $(z,t_2)$.
        \item In the complement of $Y_1$, we blow up $Y_2$ along the vanishing of $t_2$, followed by the strict transform of $Y_4$ along the vanishing of $t_1$. The ideals are $(y,t_2)$ and $(w,t_1)$, where $w$ is the local equation for $Y_4$.
\end{itemize}
As shown in \cite[Proposition 3.1.5]{CT}, the two blow-ups in each corner commute.
\end{definition}
We have just described 4 different blow-ups on the 4 corners of $X_0$. We will show that these glue together to something projective (which therefore defines a global blow-up; birational morphisms between projective spaces are blow-ups). The above described blow-ups are shown as subdivisions of the dual complex in Figure \ref{choice of subdiv}.

\begin{proposition}\label{gluing blow-ups}
    The above four blow-ups glue together.
\end{proposition}

\begin{proof}
    Figure \ref{choice of subdiv} essentially proves this statement.  In order to understand why this is the case, let us first unravel our assumptions. The four blow-ups of Definition \ref{4 BUs} are given locally as blow-ups of $Y_i$ components. In the tropical picture, as seen in the first picture of Figure \ref{choice of subdiv}, this means that the two red vertices added by the local blow-up of $Y_i$ are equidistant to $Y_i$. Together with $Y_i$ they form an equilateral triangle, and the red edge connecting them is parallel to a black edge.
    \begin{remark}
        Note here that because of the way we have drawn the picture, all triangles do not appear equilateral and therefore all added red edges are not parallel to black edges. In a true picture, these edges would be parallel to a black edge as they come, locally, from blowing up one of the irreducible components.
    \end{remark}
    The fact that these blow-ups glue together is equivalent to the new (red) edges joining up. The picture on the right hand side of Figure \ref{choice of subdiv} shows that through adding parallel edges in this way, we really can get them to join up, by showing that the pink and green distances match up between all original black triangles.
    The colours of the edges encode the blow-up data in the following way: an edge from $Y_i$ to $Y_j$, which carries an exceptional vertex $v$, with pink from $Y_i$ to $v$ and green from $v$ to $Y_j$, means that $Y_i$ is blown-up along $t_1$ and $Y_j$ is blown-up along $t_2$. The vertex $v$ is the exceptional component resulting from these blow-ups. The green length corresponds to $t_1$ and the pink to $t_2$.
\end{proof}

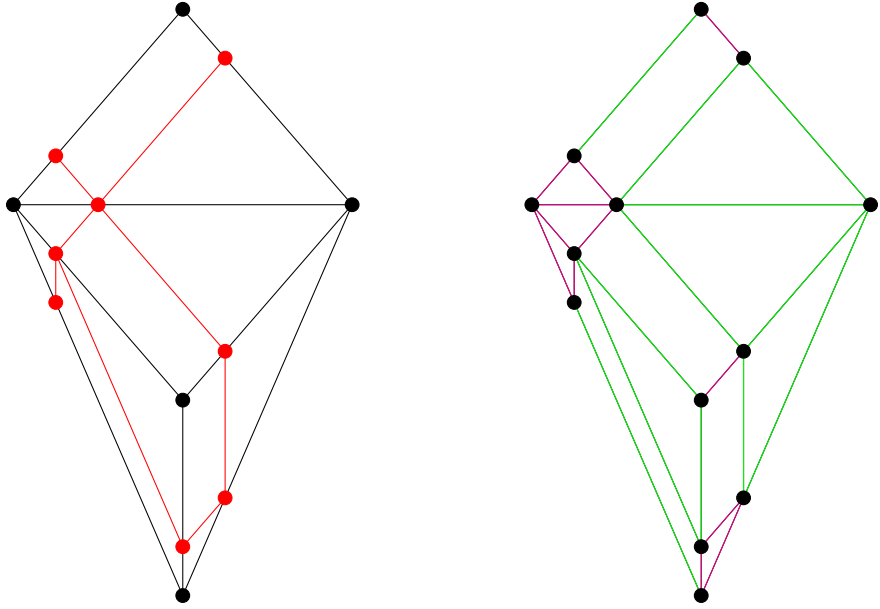
\begin{figure}
    \centering
    \begin{tikzpicture}[scale=1.3]
        \draw
        %123
        (-1.732, 2) -- (0,0)       
        (-1.732, 2) -- (1.732, 2)
        (0,0) -- (1.732, 2)
        %124
        (-1.732,2) -- (0,4)
        (1.732,2) -- (0,4)
        %134
        (0,0) -- (0,-2)
        (-1.732,2) -- (0,-2)
        (1.732,2) -- (0,-2)
        ;
        \draw[red] (-0.866,2) -- (-1.299,1.5);
        \draw[red] (-0.866,2) -- (0.433,0.5);
        \draw[red] (-0.866,2) -- (-1.299,2.5);
        \draw[red] (-0.866,2) -- (0.433,3.5);
        \draw[red] (-1.299,1.5) -- (-1.299,1);
        \draw[red] (-1.299,1.5) -- (0,-1.5);
        \draw[red] (0.433,-1) -- (0,-1.5);
        \draw[red] (0.433,-1) -- (0.433,0.5);
        
        %vertices
        \filldraw[black] (0,0) circle (2pt) ;
        \filldraw[black] (-1.732,2) circle (2pt) ;
        \filldraw[black] (1.732,2) circle (2pt) ;
        \filldraw[black] (0,4) circle (2pt) ;
        \filldraw[black] (0,-2) circle (2pt) ;
        \filldraw[red] (-0.866,2) circle (2pt) ;
        \filldraw[red] (-1.299,1.5) circle (2pt) ;
        \filldraw[red] (0.433,0.5) circle (2pt) ;
        \filldraw[red] (0.433,3.5) circle (2pt) ;
        \filldraw[red] (-1.299,2.5) circle (2pt) ;
        \filldraw[red] (-1.299,1) circle (2pt) ;
        \filldraw[red] (0,-1.5) circle (2pt) ;
        \filldraw[red] (0.433,-1) circle (2pt) ;

     \end{tikzpicture}
     \hspace{2cm}
     \begin{tikzpicture}[scale=1.3]
        \draw
        %123
        (-1.732, 2) -- (0,0)       
        (-1.732, 2) -- (1.732, 2)
        (0,0) -- (1.732, 2)
        %124
        (-1.732,2) -- (0,4)
        (1.732,2) -- (0,4)
        %134
        (0,0) -- (0,-2)
        (-1.732,2) -- (0,-2)
        (1.732,2) -- (0,-2)
        ;
        \draw[black] (-0.866,2) -- (-1.299,1.5);
        \draw[black] (-0.866,2) -- (0.433,0.5);
        \draw[black] (-0.866,2) -- (-1.299,2.5);
        \draw[black] (-0.866,2) -- (0.433,3.5);
        \draw[black] (-1.299,1.5) -- (-1.299,1);
        \draw[black] (-1.299,1.5) -- (0,-1.5);
        \draw[black] (0.433,-1) -- (0,-1.5);
        \draw[black] (0.433,-1) -- (0.433,0.5);
        %Y_1
        \draw[magenta] (-1.732,2) -- (-1.299,2.5);
        \draw[magenta] (-1.732,2) -- (-0.866,2);
        \draw[magenta] (-1.732,2) -- (-1.299,1.5);
        \draw[magenta] (-1.732,2) -- (-1.299,1);
        %Y_2
        \draw[green] (-0.866,2) -- (1.732,2);
        \draw[green] (0.433,3.5) -- (1.732,2);
        \draw[green] (0.433,0.5) -- (1.732,2);
        \draw[green] (0.433,-1) -- (1.732,2);
        %Y_3
        \draw[green] (0,0) -- (-1.299,1.5);
        \draw[green] (0,0) -- (0,-1.5);
        \draw[magenta] (0,0) -- (0.433,0.5);
        %Y_4
        \draw[green] (0,-2) -- (-1.299,1);
        \draw[magenta] (0,-2) -- (0.433,-1);
        \draw[magenta] (0,-2) -- (0,-1.5);
        \draw[green] (0,4) -- (-1.299,2.5);
        \draw[magenta] (0,4) -- (0.433,3.5);
        %inner pieces
        \draw[green] (0.433,3.5) -- (-0.866,2);
        \draw[magenta] (-0.866,2) -- (-1.299,1.5);
        \draw[magenta] (-1.299,1) -- (-1.299,1.5);
        \draw[magenta] (-0.866,2) -- (-1.299,2.5);
        \draw[magenta] (0,-1.5) -- (0.433,-1);
        \draw[green] (0.433,0.5) -- (0.433,-1);
        \draw[green] (0,-1.5) -- (-1.299,1.5);
        \draw[green] (-0.866,2) -- (0.433,0.5);

        %vertices
        \filldraw[black] (0,0) circle (2pt) ;
        \filldraw[black] (-1.732,2) circle (2pt) ;
        \filldraw[black] (1.732,2) circle (2pt) ;
        \filldraw[black] (0,4) circle (2pt) ;
        \filldraw[black] (0,-2) circle (2pt) ;
        \filldraw[black] (-0.866,2) circle (2pt) ;
        \filldraw[black] (-1.299,1.5) circle (2pt) ;
        \filldraw[black] (0.433,0.5) circle (2pt) ;
        \filldraw[black] (0.433,3.5) circle (2pt) ;
        \filldraw[black] (-1.299,2.5) circle (2pt) ;
        \filldraw[black] (-1.299,1) circle (2pt) ;
        \filldraw[black] (0,-1.5) circle (2pt) ;
        \filldraw[black] (0.433,-1) circle (2pt) ;

     \end{tikzpicture}
    \caption{Tropical pictures of our chosen subdivision.}
    \label{choice of subdiv}
\end{figure}

\begin{example}
    To help with the intuition of the above proof, we present an example of a subdivision that does not glue everywhere. Let us assume now that instead of the local blow-ups described in \ref{4 BUs}, we considered the following four blow-ups.
    \begin{itemize}
        \item In the complement of $Y_4$, we blow up $Y_1$ along the vanishing of $t_1$, followed by the strict transform of $Y_2$ along the vanishing of $t_2$ as before. In local coordinates, the ideals of the blow-ups are $(x,t_1)$ and $(y,t_2)$.
        \item In the complement of $Y_3$, we also blow up $Y_1$ along the vanishing of $t_1$, followed by the strict transform of $Y_2$ along the vanishing of $t_2$ as before. The blow-ups ideals are the same as above.
        \item In the complement of $Y_2$, we blow up $Y_4$ along the vanishing of $t_1$, followed by the strict transform of $Y_3$ along the vanishing of $t_2$. The ideals of the blow-ups are $(w,t_1)$ and $(z,t_2)$.
        \item Similarly, in the complement of $Y_1$, we blow up $Y_3$ along the vanishing of $t_2$, followed by the strict transform of $Y_4$ along the vanishing of $t_1$. The ideals are $(z,t_2)$ and $(w,t_1)$.
    \end{itemize}
    In each triangle of \ref{flattened picture}, we have produced a subdivision of the type described in \cite{CT}, as before. Moreover, this configuration may seem more symmetric, as it reflects in the $Y_1\cap Y_2$ edge and the $Y_2\cap Y_3$ edge. This can be seen in Figure \ref{bad subdivision}. However, as this Figure shows, if we proceed in this way, along the $Y_2\cap Y_3$ edge, i.e.\ in the complement of the $Y_1$ and $Y_4$ components, the two $\PP^1$ bundles produced do not glue. In terms of equations, this can be seen as follows:
    \begin{itemize}
        \item In the complement of $Y_1$, we have blow-up ideals given by $(z,t_2), (w,t_1)$.
        \item In the complement of $Y_4$, we have blow-up ideals given by $(x,t_1)$ and $(y,t_2)$.
    \end{itemize}
    In the complement of both $Y_1$ and $Y_4$ (x and w are nonzero), this leads to the equations
    \begin{align*}
        y_0t_2 &= y_1 z, \\
        y_0t_2 &= y_1 y.
    \end{align*} 
    It is clear that the $\PP^1$ bundles we locally defined do not glue in this locus. 
    In the picture this shows up in the following way. The two nodes on the edge $Y_2\cap Y_3$ of Figure \ref{bad subdivision} result from blow-ups of $Y_3$ along $t_2$ and $Y_2$ along $t_2$. They must therefore be a red distance away from $Y_2$ and a red distance away from $Y_3$ respectively (recall the red corresponds to $t_2$). As a consequence, the two nodes cannot be the same, meaning that the $\PP^1$-bundles do not glue in this locus.
\end{example}
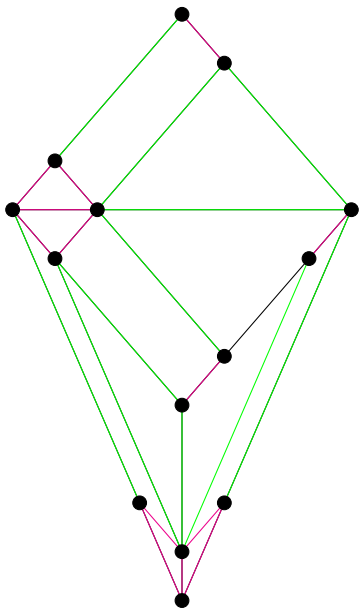
\begin{figure}
    \centering
    \begin{tikzpicture}[scale=1.3]
        \draw
        %123
        (-1.732, 2) -- (0,0)       
        (-1.732, 2) -- (1.732, 2)
        (0,0) -- (1.732, 2)
        %124
        (-1.732,2) -- (0,4)
        (1.732,2) -- (0,4)
        %134
        (0,0) -- (0,-2)
        (-1.732,2) -- (0,-2)
        (1.732,2) -- (0,-2)
        ;
        \draw[black] (-0.866,2) -- (-1.299,1.5);
        \draw[black] (-0.866,2) -- (0.433,0.5);
        \draw[black] (-0.866,2) -- (-1.299,2.5);
        \draw[black] (-0.866,2) -- (0.433,3.5);
        %\draw[black] (-1.299,1.5) -- (-1.299,1);
        \draw[black] (-1.299,1.5) -- (0,-1.5);
        %\draw[black] (0.433,-1) -- (0,-1.5);
        %\draw[black] (0.433,-1) -- (0.433,0.5);
        %Y_1
        \draw[magenta] (-1.732,2) -- (-1.299,2.5);
        \draw[magenta] (-1.732,2) -- (-0.866,2);
        \draw[magenta] (-1.732,2) -- (-1.299,1.5);
        %\draw[magenta] (-1.732,2) -- (-1.299,1);
        %Y_2
        \draw[green] (-0.866,2) -- (1.732,2);
        \draw[green] (0.433,3.5) -- (1.732,2);
        %\draw[green] (0.433,0.5) -- (1.732,2);
        \draw[green] (0.433,-1) -- (1.732,2);
        %Y_3
        \draw[green] (0,0) -- (-1.299,1.5);
        \draw[green] (0,0) -- (0,-1.5);
        \draw[magenta] (0,0) -- (0.433,0.5);
        %Y_4
        %\draw[green] (0,-2) -- (-1.299,1);
        \draw[magenta] (0,-2) -- (0.433,-1);
        \draw[magenta] (0,-2) -- (0,-1.5);
        \draw[green] (0,4) -- (-1.299,2.5);
        \draw[magenta] (0,4) -- (0.433,3.5);
        %inner pieces
        \draw[green] (0.433,3.5) -- (-0.866,2);
        \draw[magenta] (-0.866,2) -- (-1.299,1.5);
        %\draw[magenta] (-1.299,1) -- (-1.299,1.5);
        \draw[magenta] (-0.866,2) -- (-1.299,2.5);
        \draw[magenta] (0,-1.5) -- (0.433,-1);
        %\draw[green] (0.433,0.5) -- (0.433,-1);
        \draw[green] (0,-1.5) -- (-1.299,1.5);
        \draw[green] (-0.866,2) -- (0.433,0.5);
        %new edges in bad subdivision
        \draw[green] (-1.732,2) -- (-0.433,-1);
        \draw[magenta] (0,-2) -- (-0.433,-1);
        \draw[magenta] (0,-1.5) -- (-0.433,-1);
        \draw[green] (0,-1.5) -- (1.299,1.5);
        \draw[magenta] (1.299,1.5) -- (1.732, 2);

        %vertices
        \filldraw[black] (0,0) circle (2pt) ;
        \filldraw[black] (-1.732,2) circle (2pt) ;
        \filldraw[black] (1.732,2) circle (2pt) ;
        \filldraw[black] (0,4) circle (2pt) ;
        \filldraw[black] (0,-2) circle (2pt) ;
        \filldraw[black] (-0.866,2) circle (2pt) ;
        \filldraw[black] (-1.299,1.5) circle (2pt) ;
        \filldraw[black] (0.433,0.5) circle (2pt) ;
        \filldraw[black] (0.433,3.5) circle (2pt) ;
        \filldraw[black] (-1.299,2.5) circle (2pt) ;
        %\filldraw[black] (-1.299,1) circle (2pt) ;
        \filldraw[black] (0,-1.5) circle (2pt) ;
        \filldraw[black] (0.433,-1) circle (2pt) ;
        %new vertices of bad subdivision
        \filldraw[black] (-0.433,-1) circle (2pt) ;
        \filldraw[black] (1.299,1.5) circle (2pt);

     \end{tikzpicture}
    \caption{A subdivision that does not glue.}
    \label{bad subdivision}
\end{figure}

\begin{definition}\label{Definition : X[1]}
    Reusing the notation of Section \ref{section: background}, we denote by $X[1]$ the space resulting from doing all four modifications of \ref{4 BUs} on $X \times_{\AAA^1}\AAA^2$ and gluing them together. By induction, it is clear that the same argument can be extended to further subdivisions to build spaces $X[n]$ for $n>1$. A precise local description of $X[n]$ was given in Section \ref{Section : expanded degenerations}.
\end{definition}

\begin{remark}\label{Remark : general degen}
    In this paper, we focus on the quartic degeneration and cube degeneration only to demonstrate the two possibilities for moving from the local situation to a general type III degeneration of K3 surfaces. Either the components of the central fibre can be consistently labelled as in Section \ref{Section : Cube} or we must come up with a global subdivision, as in Section \ref{Section : Quartic}. For example, suppose we took (a resolution) of the degeneration of a $(3,2)$-hypersurface in $\PP^2 \times \PP^1$ to the toric boundary of $\PP^2 \times \PP^1$. Then a global subdivision for the central fibre giving rise to $X[1]$ is displayed in Figure \ref{fig:threetwo}.
\end{remark}

\begin{figure}[h]
    \centering
    \includegraphics[width=0.3\linewidth]{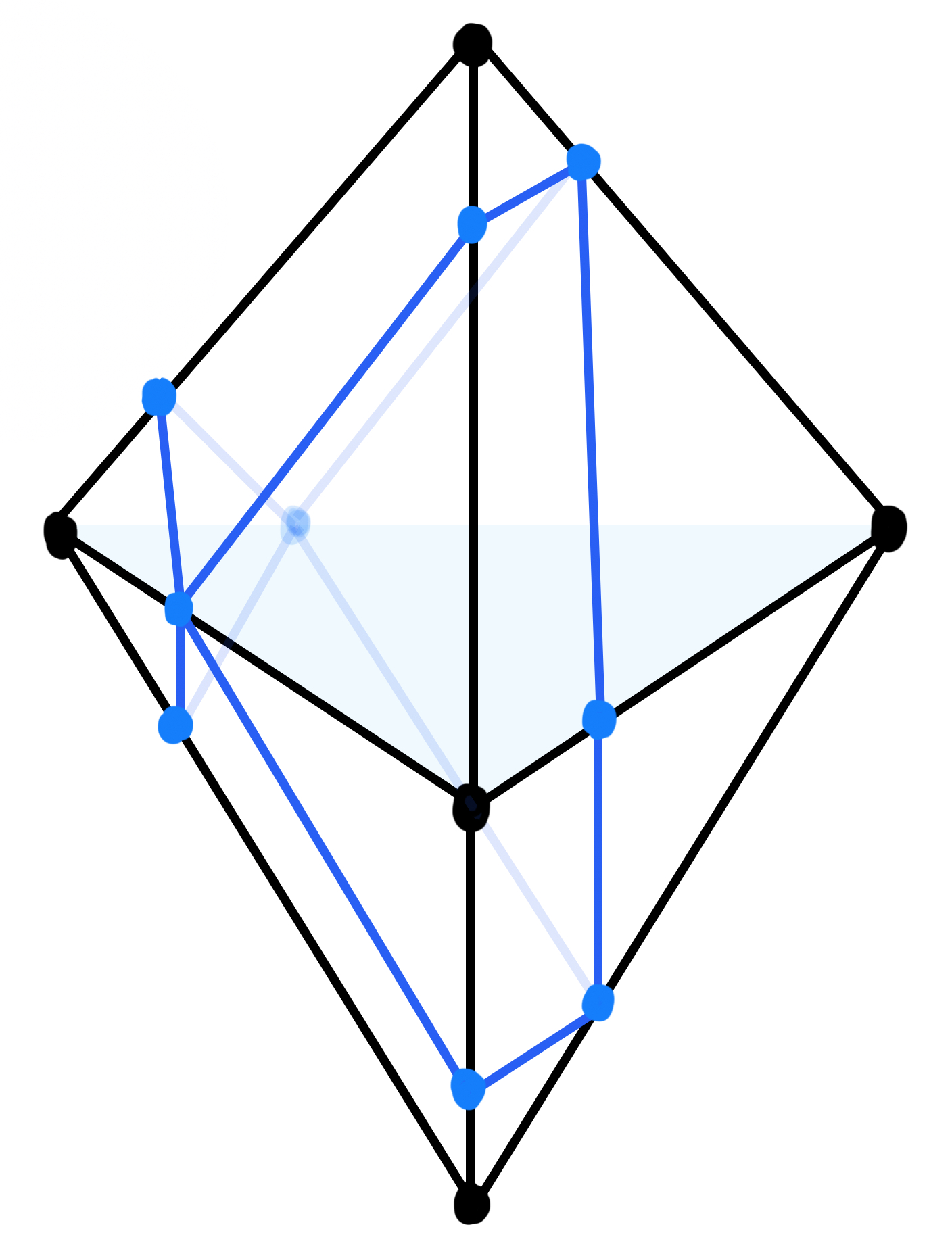}
    \caption{Subdivision of the dual complex for the central fibre of the degeneration described in Remark \ref{Remark : general degen}}
    \label{fig:threetwo}
\end{figure}

\subsection{Projectivity}

The goal of this section is to show that the expansions introduced in Section \ref{Section : Gluing blow-ups} are projective. Specifically, we show that the morphism $X[n] \rightarrow \AAA^{n+1}$ is projective.

For the cube degeneration this is immediate from \cite[Proposition 3.1.7]{CT}. For the quartic degeneration we explain how the tropical picture in Figure \ref{choice of subdiv} encodes a way to construct $X[n]$ from $X \times_{\AAA^1} \AAA^{n+1}$, in line with toric geometry. Toric varieties correspond to fans; a collection of rational polyhedral cones glued along faces, living inside an ambient vector space. Moreover, subdivisions of fans correspond to birational modifications of the associated toric variety. In our situation, neither $X[n]$ nor $X \times_{\AAA^1} \AAA^{n+1}$ are toric varieties. However, there is a \emph{cone complex} $\Sigma_n$ associated to $X \times_{\AAA^1} \AAA^{n+1}$. In other words, a collection of rational polyhedral cones without the global embedding into a vector space. Once again, subdivisions of this cone complex produce birational modifications \cite[Section 1.3]{MR}. The tropical picture in Figure \ref{choice of subdiv} encodes a subdivision of this cone complex which produces the birational modiciation $X[n] \rightarrow X \times_{\AAA^1} \AAA^{n+1}$. Using this perspective, the same combinatorial criterion of projectivity in toric geometry can also be used to show that the corresponding morphism is projective. 

\subsubsection{Cone complex of $X \times_{\AAA^1} \AAA^{n+1}$}

The cone complex $\Sigma_n$ is itself the fibre product of the cone complexes $\Sigma(X) \times_{\Sigma(\AAA^1)} \Sigma(\AAA^{n+1})$ associated to $X,\AAA^{n+1}, \AAA^1$. They are each given by the cones over the dual complexes with respect to the origin in $\AAA^1$, the coordinate hyperplanes in $\AAA^{n+1}$ and the special fibre of the degeneration in $X$ respectively. Explicitly these are:

\begin{itemize}
    \item The cone complex $\Sigma(\AAA^1)$ is the fan of the toric variety $\AAA^1$, a single ray $\RR_{\geq 0}$.
    \item The cone complex $\Sigma(\AAA^{n+1})$ is the fan of the toric variety $\AAA^{n+1}$, an $n+1$ dimensional maximal cone $\RR^{n+1}_{\geq 0}$.
    \item The cone complex $\Sigma(X)$ is the cone over the tetrahedron. It consists of four $3$-dimensional cones, each one glued to the three others along one of the three two-dimensional faces.
\end{itemize}

The morphisms are as follows:

\begin{itemize}
    \item $\Sigma(\AAA^{n+1}) \rightarrow \Sigma(\AAA^1)$ is the morphism $(\tau_1,\dots,\tau_n) \mapsto \tau_1 + \dots + \tau_n$
    \item On each three dimensional cone the morphism $\Sigma(X) \rightarrow \Sigma(\AAA^1)$ is $(a_1,a_2, a_3) \mapsto a_1 + a_2 + a_3$
\end{itemize}

One can visualise the cone complex $\Sigma_n$, see Figure \ref{fig:Sigma}, as a fibration over $\Sigma(\AAA^{n+1})$ where the fibre over the origin is a point and the fibre over all other points is the (boundary of the) tetrahedron with edge lengths $\tau_1 + \dots +\tau_n$.

\begin{figure}\hspace{-2cm}
    \centering
    \includegraphics[width=1\linewidth]{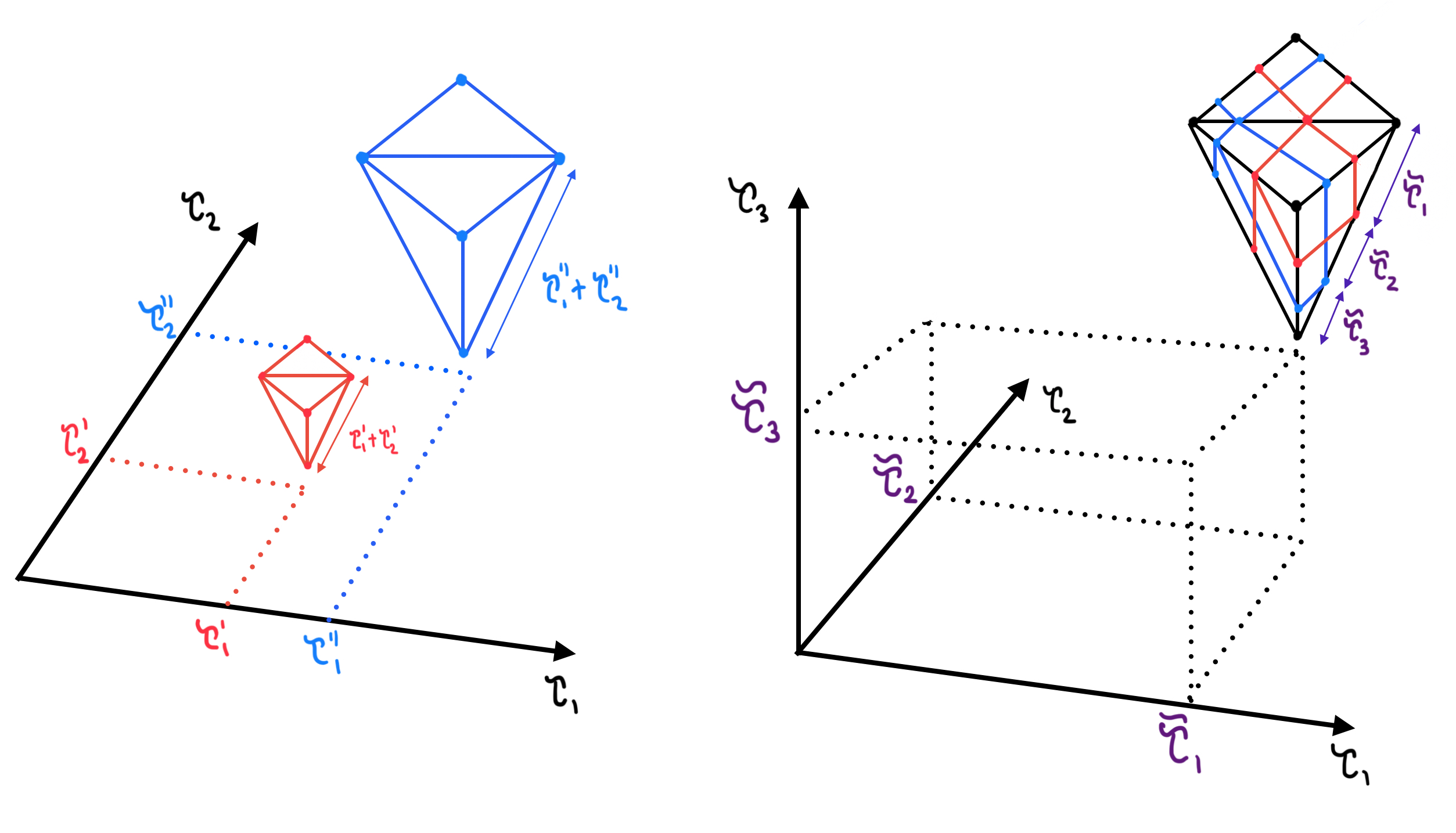}
    \caption{Cone complex $\Sigma_1$ (left) and the subdivided cone complex $\widetilde{\Sigma_2}$.}
    \label{fig:Sigma}
\end{figure}

\subsubsection{Subdivision and cone complex of $X[n]$}

This allows us to name a subdivision of the cone complex $\Sigma_n$ by giving a polyhedral subdivision of the tetrahedron fibres. Our subdivision is given by Figure \ref{fig:Sigma}. The fibre over $(\tau_1,\dots,\tau_n)$ still consists of a tetrahedron with edge lengths $\tau_1 + \dots + \tau_n$. However now the tetrahedron is polyhedrally subdivided as in Figure \ref{fig:Sigma}; on each tetrahedron fibre over $(\tau_1,\dots,\tau_n)$ we first make the polyhedral subdivision as in Figure \ref{choice of subdiv}, determined by adding a vertex along the edge connecting the vertices corresponding to $Y_1$ and $Y_2$ at distance $\tau_1$ from the $Y_2$ vertex. Then we add an identical subdivision determined by adding another vertex at distance $\tau_2$ from our previous vertex and so on until we have introduced $n$ new vertices to the edge. Moreover, the cones of the new cone complex $\widetilde{\Sigma}$ are distinguished by the varying $\tau_1,\dots,\tau_n$ but remaining within a fixed polyhedral region of the subdivided tetrahedron. In particular, the number of maximal $(n+3)$-dimensional cones is equal to the number of maximal polyhedral regions of the subdivided tetrahedron. Moreover, this subdivision produces the birational modification $X[n] \rightarrow X \times_{\AAA^1} \AAA^{n+1}$ described by the local blow-ups from Definition \ref{Definition : X[1]}.

%\begin{figure}
 %   \centering
    %\includegraphics[width=0.6\linewidth]{IMG_1811.jpeg}
    %\caption{The subdivided cone complex $\widetilde{\Sigma}$ for $n=2$.}
    %\label{fig:Sigmatilde}
%\end{figure}

In toric geometry, piecewise-linear functions on the fan correspond to torus equivariant Cartier divisors on the associated toric variety. In other words, real-valued functions that are linear on each cone of the fan. Moreover, ampleness of the associated line bundle amounts to this function being \emph{strictly convex} \cite[Theorem 6.1.14]{cox2011toric}, i.e.\ convex and such that the function is different on different cones. There is a similar criterion to check whether a toric morphism is projective \cite[Theorem 7.2.11]{cox2011toric}. We will need a similar result for cone complexes. 

\begin{lemma}
    Let $\phi : \widetilde{\Sigma} \rightarrow \Sigma$ be a subdivision. Suppose $\widetilde{\Sigma}$ admits a piecewise linear function $\varphi$, which is strictly convex on each $\phi^{-1}(\sigma)$ for $\sigma$ a cone of $\Sigma$. Then the associated birational modification is projective.
\end{lemma}

\begin{proof}
   % \Qaasim{this is not ready/complete} 
   The condition implies that the morphism of toric varieties $X_{\widetilde{\Sigma}} \rightarrow X_{\Sigma}$ is projective. We are abusing notation in the sense that the toric variety is determined by a fan, rather than just a cone complex structure, but the choice does not matter. On the other hand, unraveling the definition of a birational induced by a subdivision of cone complexes, \cite[Definition 1.3.2]{MR}, tells us that the birational modification $X[n] \rightarrow X \times_{\AAA^1} \AAA^{n+1}$ fits into the cartesian diagram
\begin{equation*}
    \begin{tikzcd}
{X[n]} \arrow[d] \arrow[r]             & {[X_{\widetilde{\Sigma_n}}/T_{\widetilde{\Sigma_n}}]} \arrow[d] \\
X \times_{\AAA^1} \AAA^{n+1} \arrow[r] & {[X_{\Sigma_n}/T_{\Sigma_n}]}                             
\end{tikzcd}
\end{equation*}
where $T_{\widetilde{\Sigma_n}},T_{\Sigma_n}$ denote the dense tori inside $X_{\widetilde{\Sigma_n}}, X_{\Sigma_n}$. Since the right hand morphism is representable, a relatively ample line bundle on $[X_{\widetilde{\Sigma_n}}/T_{\widetilde{\Sigma_n}}]$ is such that the pullback via any morphism of schemes gives a relatively ample line bundle. The result then follows from the fact that this is equivalent to an equivariant relatively ample line bundle on the prequotient. 
\end{proof}

\begin{lemma}
    The morphism $X[1] \rightarrow X \times_{\AAA^1} \AAA^2$ is a projective morphism.
\end{lemma}

\begin{proof}
It suffices to produce a strictly convex piecewise linear function on the slice  $\{\tau_1 + \tau_2 =1\}$ as $\widetilde{\Sigma}_n$ (and $\Sigma_n$) are given by the cone over this slice. As none of the boundaries between regions are contained entirely in a hyperplane given by fixing $\tau_2 =: \tau$, it is enough to produce a such a piecewise linear function for fixed $\tau_2$. The (negative) of a function presented as the minimum of a collection of linear functions is convex. As the function needs to be strictly convex on the preimage of each cone in $\Sigma_1$, we present such a function on each face of the tetrahedron. 

\begin{itemize}
    \item On the $Y_1,Y_2,Y_3$ face, with $c$ defined as the coordinate direction from $Y_3$ to $Y_2$ and $q$ coordinate direction from $Y_2$ to $Y_1$, the function is given by $\min\{2c-q,-c+2q + 3 - 3\tau,2c-3q+2\tau\}$
    \item On the $Y_4,Y_3,Y_2$ face, with coordinate $c$ defined as the coordinate direction from $Y_2$ to $Y_3$ and $q$ defined as the coordinate direction from $Y_3$ to $Y_4$, the function is given by $\min\{-2c +q +2, c+q+2-3\tau, -2c+2+\tau\}$
\item On the $Y_1, Y_3, Y_4$ face, with $c$ defined as the coordinate direction from $Y_4$ to $Y_3$ and $q$ the coordinate direction from $Y_3$ to $Y_1$, the function is given by $\min\{\tau,-c+q+1,-q +2\tau\}$
\item On the $Y_1, Y_4, Y_2$ face, with $c$ defined as the coordinate direction from $Y_4$ to $Y_2$ and $q$ the coordinate direction from $Y_2$ to $Y_1$, the function is given by $\min\{2c -2q + \tau, -2c + 2q + 4-3\tau, 2c-3q +2\tau\}$.
\end{itemize}
\end{proof}

\subsubsection{Generalising to $X[n]$.}
We now use the fact that $X[1] \rightarrow X \times_{\AAA^1} \AAA^{2}$ to show that $X[n] \rightarrow X \times_{\AAA^1} \AAA^{n+1}$.

\begin{proposition}
    $X[n]\to X\times_{\AAA^1}\AAA^{n+1}$ is a projective morphism.
\end{proposition}
\begin{proof}

We can do this by induction. For simplicity we show the $n=2$ case. 
Consider the cartesian diagrams 
\begin{equation*}
    \begin{tikzcd}
{X[1] \times_{\AAA^1} \AAA^2} \arrow[d] \arrow[r] & X \times_{\AAA^1} \AAA^3 \arrow[r] \arrow[d] & \AAA^3 \arrow[d, "{(t_1t_2, t_3)}"] \arrow[r, "{(t_1,t_2)}"] & \AAA^2 \arrow[d, "m"] \\
{X[1]} \arrow[r]           & X \times_{\AAA^1} \AAA^2 \arrow[r]           & \AAA^2 \arrow[r, "\mathsf{pr}_1"]                                           & \AAA^1               
\end{tikzcd}
\end{equation*}

    We have shown that the morphism $X[1] \to X\times_{\AAA^1}\AAA^2$ is projective. As it is birational, it is the blow up in some ideal sheaf $\cI$. Now consider the pullback of $\cI$ from $X \times_{\AAA^1} \AAA^2$ to ${X[1] \times_{\AAA^1} \AAA^2}$. The blow up in this ideal sheaf is therefore projective over $X[1] \times_{\AAA^1} \AAA^2$ and hence over $X \times_{\AAA^1} \AAA^3$ as the latter morphism is the base change of a projective morphism. As the morphism $\AAA^3 \rightarrow \AAA^2$ is given by $(t_1,t_2,t_3 )\mapsto (t_1t_2,t_3)$, we can see that the blow-up in the pull back ideal sheaf is in fact $X[2]$. Indeed, this can be seen by looking at the equations for the blow-up in the local patches.

    We may iterate this process by taking the fibre product $X[2]\times_{\AAA^1}\AAA^2$ given by the morphism $(t_3,t_4)\in \AAA^2 \mapsto (t_3t_4)\in\AAA^1$ and considering the following morphisms $\AAA^4\to \AAA^2$: $(t_1,t_2,t_3,t_4)\mapsto (t_1,t_2t_3t_4)$, $(t_1,t_2,t_3,t_4)\mapsto (t_1t_2,t_3t_4)$ and $(t_1,t_2,t_3,t_4)\mapsto (t_1t_2t_3,t_4)$
\end{proof}

\begin{corollary}
    The morphism $X[n] \rightarrow \AAA^{n+1}$ is projective.
\end{corollary}

\subsection{The semistable degeneration}

In this section we prove Theorem \ref{Theorem : main}. Namely, that taking LW stable length $m$ zero-dimensional subschemes in the above expansions leads to a good degeneration of hyperkählers with general fibre $\mathrm{Hilb}^m(\mathrm{K}3)$. 

The family $X[n]$ lives over the $(n+1)$-dimensional base $\AAA^{n+1}$ and contains copies of the same fibre. In order to reduce back down to a 1-dimensional base we quotient out by the torus action, locally described in Section \ref{Section : torus action}.

\begin{lemma}
    The action of the torus $\GG_m^n$ described in Section \ref{Section : torus action}, extends to an action on $X[n]$ for either degeneration.
\end{lemma}

\begin{proof}
    The torus action on the exceptional $\PP^1$ coordinates is induced by its action on the base $\AAA^{n+1}$ and the equations of the blow up. This makes it clear that $\GG_m^n$ acts on all exceptional coordinates in either of the expanded degeneration we study here. What remains to be shown is that the action on different affine patches glues in a compatible way.

    In the case of the degree 6 K3 surfaces, we blow up simultaneously both disjoint components labelled $Y_1$, followed by the same operation on the disjoint pair of components labelled $Y_2$. The first blow-up results in 4 new exceptional components appearing around each $Y_1$ component, given by $\PP^1 $-bundles over the intersection curves of $Y_1$ with its neighbouring components. We act on each $\PP^1$ by $\GG_m$. We represent this by red arrows through each exceptional component to signify the direction of the action as in Figure \ref{fig:global torus action}. When we do the blow-ups for the $Y_2$ component, we recall that one of the exceptional components is both a $\Delta_1$ and a $\Delta_2$ component. The same $\GG_m$ therefore acts on it, and as can be seen from the equations of the blow-ups, when a $\tau\in \GG_m$ acts on this $\Delta_1$, the corresponding $\tau^{-1}$ acts on this $\Delta_2$. We represent this in Figure \ref{fig:global torus action} by green arrows through the $\Delta_2$ exceptional components pointing outwards from the $Y_2$ component. A green arrow is to be understood by the above explanation as the inverse of a red arrow. We could also have drawn red arrows pointing inwards towards $Y_2$. Compatibility of the gluing for the torus action then comes down to asking that no exceptional component has arrows of the same colour pointing in different directions, or, equivalently, arrows of different colours pointing in the same direction.
    
    The colour-coding of Figure \ref{fig:global torus action} sets us up for the proof in the case of a maximally degenerate K3 surface of degree four. We consider again Figure \ref{choice of subdiv}. The same criterion about red and green arrows needs to apply here to the exceptional components, shown as red vertices in the left-hand picture. A study of the equations of the blow-ups shows that for each such exceptional node, the corresponding geometric component has red and green arrows going through it in the following way. Draw on each green or red edge going from a $Y_i$ vertex to an exceptional vertex an arrow pointing towards the exceptional; this describes the group action on that geometric exceptional component in the sense of the previous paragraph. On edges between two exceptional vertices, the arrow should be perpendicular to the edge and of the same colour and direction as the arrows entering both exceptional nodes. These edges correspond, on the geometric side, to a $\PP^1$ over a point. A red arrow means that the coordinates $[u:v]$ of this $\PP^1$ are acted upon by $\tau\in\GG_m$ by $[\tau u:v]$. Note that this is equivalent to a green arrow in the opposite direction through this edge, which signifies that the $\PP^1$ with coordinates $[v:u]$ is acted upon by $[v:\tau u]$. A green arrow in the same direction as the red would be incompatible because it would mean $[u:v]$ is acted upon by $[u:\tau v]$. A quick sketch of the geometric picture, adding green and red arrows according to this rule, shows that our above compatibility condition is satisfied.
\end{proof}

\begin{figure}
    \centering
    \begin{tikzpicture}[scale= 2]
  % First square
      \draw (0,0) -- (1,0) -- (1,1) -- (0,1) -- cycle;
      
      % Second square (adjacent to the right)
      \draw (1.5,0) -- (2.5,0) -- (2.5,1) -- (1.5,1) -- cycle;
      \draw (1,1) -- (1.25,1.25) -- (1.5,1);
      \draw (1.25,1.25) -- (-0.25, 1.25) -- (-0.25,-0.25);
      \draw (-0.25,-0.25) -- (1.25, -0.25) -- (2.75, -0.25) -- (2.75, 1.25) -- (1.25,1.25);
      \draw (-0.25, 1.25) -- (0,1);
      \draw (0,0) -- (-0.25, -0.25)
      (1,0) -- (1.25, -0.25)
      (1.25, -0.25) -- (1.5, 0)
      (2.5, 0) -- (2.75, -0.25)
      (2.5, 1) -- (2.75, 1.25)
      ;
      \draw[->, color=red, thick] (1,0.4) -- (1.5,0.4);
      \draw[<-, color=green, thick] (1,0.6) -- (1.5,0.6);
      \draw[->, color=red, thick] (0.5,1) -- (0.5,1.25);
      \draw[->, color=red, thick] (0,0.5) -- (-0.25,0.5);
      \draw[->, color=red, thick] (0.5, 0) -- (0.5, -0.25);
      \draw[->, color=green, thick] (2,0) -- (2,-0.25);
      \draw[->, color=green, thick] (2.5, 0.5) -- (2.75, 0.5);
      \draw[->, color=green, thick] (2,1) -- (2, 1.25);
    \end{tikzpicture}
    \caption{Partial picture of the torus action on a modified maximally degenerate K3 of dimension 6.}
    \label{fig:global torus action}
\end{figure}
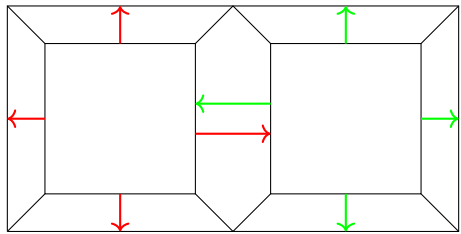

\begin{definition}
    For either degeneration, let $\frX = \lim_{\to} [X[n]/\sim]$ be the family over the stack of expansions $\mathfrak{C}$.
\end{definition}

\begin{definition}
    Let $\frM_{\mathrm{LW}}^m$ be the stack of LW stable length $m$ zero-dimensional subschemes in $\frX$.
\end{definition}

\begin{theorem}\label{Theorem : main}
    The stack of families of stable length $m$ zero-dimensional subschemes in $\mathfrak{X}$ over $C$ is Deligne-Mumford, proper over $C$ gives a good degeneration over $C$.
\end{theorem}
\begin{proof}
The proof of properness and finite stablisers will follow exactly as in \cite[Section 6.1]{CT}. On the other hand, in order to show semistability we will use the argument of \cite[Lemma 5.4, Lemma 7.3]{gulbrandsen2021geometry}. Here, the authors use the fact that the stack can be presented as a global quotient stack by the torus when restricting from LW to the GIT stability condition. In our situation the analogue is that if a length $m$ zero dimensional stable subscheme in $\mathrm{Hilb}^m(X[2m]/\AAA^{2m+1})$ has maximal limits, see Lemma \ref{lemma: max limits} and Corollary \ref{cor: max limits}, then it is the unique representative of the element in $\frM_{\mathrm{LW}}^m$ amongst those identified by the equivalences \cite[Section 5.2]{CT}. Consequently, one can use \cite[Lemma 6.2.1]{CT} to identify $\frM^{m}_{\mathrm{LW}}$ with a global quotient stack by the torus $\GG_m^{2m}$. Moreover, as the orbits are closed, we can use Luna's \'etale slice theorem to produce an atlas which is semistable. This follows from the fact that the morphism from the locus with maximal limits inside $H^m_{[2m],\mathrm{LW}}$ to $\AAA^{2m+1}$ is smooth, and the morphism $\AAA^{2m+1} \rightarrow \AAA^1$ is the product of the coordinates. Finally, the existence of the relative logarithmic 2-form follows from \cite[Proof of Theorem 6.4]{CT2}. 
\end{proof} 

\begin{definition}
    We say that $Z \in H^m_{[2m],\mathrm{LW}}$ has \emph{maximal limits} if each $P_i \in \mathrm{Supp}(Z)$ with multiplicity $m_i$ is in the interior of $m_i$ consecutive $\Delta_1$-components and $m_i$ consecutive $\Delta_2$ components which are unique to it. These $\Delta_2$ components must not equal any of the $\Delta_1$ when all components are expanded out (i.e.\ they must not share a $\GG_m$- action with any of the $\Delta_1$ components). In other words, $Z$ has the maximal amount of stable degenerations.
\end{definition}

Recall that the isomorphisms between objects of the stack $\mathfrak{X}$ and thus of $\mathfrak{M}_{\mathrm{LW}}^m$ were given by the torus action and some isomorphisms defined in \cite[Section 5.2]{CT}. The purpose of these isomorphisms was to identify any two fibres of $X[n] \to \AAA^{n+1}$ which had the same base codimension. Indeed, given our stability conditions, it can occur that two such fibres contain identical length $m$ 0-dimensional subschemes, leading a family of such subschemes to have more than one limit, thus breaking separatedness. These elements were identified within an equivalence class in order to avoid such an issue. In the following lemma, we will show that restricting our stability condition to only include subschemes with maximal limits allows us to remove these equivalences.
\begin{lemma}\label{lemma: max limits}
    Take any equivalence class $[(Z,\mathcal{X})]\in \mathfrak{M}_{\mathrm{LW}}^m$, where $Z$ is a LW stable length $m$ 0-dimensional subscheme of the fibre $\mathcal{X}$ of $\mathfrak{X}\to \mathfrak{C}$. Any representative of this equivalence class which lies in $\mathrm{Hilb}^m(X[2m]/\AAA^{2m+1})$ and has maximal limits has  exactly the same vanishing coordinates in $\AAA^{n+1}$.
\end{lemma}

\begin{proof}
    
    First, we will assume that $\mathcal{X}$ is a copy of $X_0$. We will prove this in the local situation on the affine patch around $Y_1\cap Y_2\cap Y_3$ and later explain how it generalises. We consider the support of $Z$ as $m$ distinct points. If a point has multiplicity $k>1$, we may treat it as $k$ points lying very close together for the purposes of this proof, as the concept of maximal limits treats them in the same way. Let $a_1$, $a_2$ and $a_3$ be the number of points lying in the $Y_1$, $Y_2$ and $Y_3 $ components of $X_0$ respectively with $a_1+a_2+a_3=m$. For $Z$ to have maximal limits each point must lie in its own associated $\Delta_1$ and $\Delta_2$ components, which do not share a $\GG_m$-action. When the subscheme is maximally degenerated, this $\Delta_1$ and $\Delta_2$ will form the $\PP^1\times\PP^1$ box in which the point lands. It will be acted on by a $\GG_m^2$ and no other points will be acted upon by the same action. This means that $Y_1$ in $X_0$ must have coordinates proportional to those of $2a_1$ different $\Delta_1=\Delta_2$ exceptional $\PP^1$ coordinates (they must be equal in this locus). Similarly, $Y_2$ in $X_0$ has coordinates proportional to those of $2a_2$ different $\Delta_1=\Delta_2$ exceptional $\PP^1$ coordinates. Now, the $a_3$ points in $Y_3$ must also be able to degenerate to distinct boxes which share no group action with the others. It follows that the coordinates of $Y_3$ must be proportional to the exceptional coordinates of $a_3$ different $\Delta_1$ components, which in $X_0$ lie in the $Y_2\cup Y_3$ locus. Similarly, $Y_3$ has coordinates proportional to $a_3$ $\Delta_2$ components which lie in the $Y_1\cup Y_2$ locus of $X_0$. In total, $Y_1$ contains $2a_1+a_3$ copies of $\Delta_1=\Delta_2$, the $Y_2$ component contains $2a_2 +a_3 $ of them, and $Y_3$ contains $2a_2+a_3$ copies of $\Delta_1$ along with $2a_1 +a_3$ copies of $\Delta_2$. These numbers sum to $2m$ $\Delta$ components. The $2m$ $\Delta$-components (i.e.\ sets of exceptional coordinates) can only be split into $Y_1$, $Y_2$ and $Y_3$ when $t_{2a_2 + a_3} =0$.
    A similar argument can be made if $\mathcal{X}$ is a different modification of $X_0$.

    We now consider a global model for $X_0$ or modifications thereof, where all points of the support of $Z$ do not lie in the same affine patch. We assign $\Delta$-components to the $Y_i$ to obtain the property of maximal limits as described above. As the local modifications were shown to glue to a global blow-up with a global torus action, in the complement of each $Y_i$, we still have three $Y_j$ components with coordinates proportional to a collection of $\Delta$ components which add up to the total number ($2m$) of distinct $\Delta$ components.
\end{proof}

\begin{remark}
    We note that a given configuration of points (viewed only as support, forgetting their subscheme structure) in $X_0$ can degenerate to different maximal limits. For $m=2$, for example, 2 points in $Y_1$ can degenerate to 3 different stable limits in the fibre above $t_1=\dots=t_{5}=0$. This choice of stable limit is encoded in the order of the $a_3$ components in $Y_1$ and $Y_2$.
\end{remark}

\begin{corollary}\label{cor: max limits}
    Every equivalence class $[(Z,\mathcal{X})]$, as defined in Lemma \ref{lemma: max limits}, has a unique representative in $H^m_{[2m],\mathrm{LW}}$ with maximal limits up to $\GG_m^{2m}$-action.
\end{corollary}
\begin{proof}
    In Lemma \ref{lemma: max limits}, we have shown that for such an equivalence class all representatives with maximal limits lie in the same scheme-theoretical fibre over $\AAA^{2m+1}$. Moreover, we have shown that the only way in which they can differ is in the choice of assignment of $\Delta$-components to points, i.e. the order of the $a_3$ components among the $a_1$ and $a_2$ (in the notation of the proof of Lemma \ref{lemma: max limits}). This choice determines different limits to which the points can degenerate, but the fibre $\mathcal{X}$ itself does not see this ordering. Each irreducible surface in $\mathcal{X}$ has coordinates proportional to a given number of exceptional coordinates. This number remains the same no matter what ordering we have placed on the $a_3$ sets of coordinates among the others. As all representatives of $[(Z,\mathcal{X})]$ have the same scheme structure on $Z$ and support within $\mathcal{X}$, the representatives differ only in terms of the scheme theroretical fibre in which they lie. There is therefore a unique representative with maximal limits for each equivalence class.
\end{proof}

\begin{remark}
    The GIT condition of \cite{GHH} cuts out exactly the subschemes with maximal limits, which is why they could write their solution as a quotient stack.
\end{remark}

The degenerations of $\mathrm{Hilb}^2(\mathrm{K3)}$ resulting from applying our above constructions to degenerations of degree 4 and 6 K3 surfaces are denoted by $Q\to \AAA^1$ and $C\to \AAA^1$ respectively.

\begin{corollary}
    The degenerations $Q\to \AAA^1$ and $C\to \AAA^1$ are good type III degenerations of hyperk\"ahler fourfolds.
\end{corollary}

\section{Two dual complex computations}\label{section:dual complex}
\subsection{General results}
Note that the dual complexes we study here are delta-complexes, as opposed to simplicial complexes, i.e.\ each simplex is not uniquely determined by its vertices.

%\Calla{Add Brown-Mazzon result?}

\begin{proposition}
    Let $Y\to \AAA^1$ be either of the degenerations of Hilbert schemes of 2 points on K3 surfaces constructed in previous sections. The dual complex $\Pi_0$ of its special fibre $Y_0$ is a 4-dimensional delta-complex whose $k$-dimensional strata correspond to components of base codimension $5-k$ in the expanded degeneration construction.
\end{proposition}
\begin{proof}
   Vertices of $\Pi_0$ correspond to 4-dimensional components of $Y_0$. The points in the interior of theses 4-dimensional components are length 2 0-dimensional subschemes with support in the interior of a $Y_i$-component of $X_0$. The intersections of these 4-dimensional components of $Y_0$ with each other are represented by edges between the vertices on the dual side. Given a configuration of points on $X_0$ defining a 4-dimensional component of $Y_0$, we can find its possible intersections with other components in the following way. Take a stable configuration of points on $X_0$ and a union of affine pieces on the surface $X_0$ contanining it. For simplicity, let us imagine this is the affine patch around the triple intersection of $Y_1$, $Y_2$ and $Y_3$. Subdivide it by placing a node in the interior of either $Y_1$ or $Y2$, but not $Y_3$. Then draw 3 half-lines from this node, parallel to each of the intersection lines of $X_0$ in such a way that we obtain Picture 1 of Figure \ref{local subdivision}. If we are considering more than one affine patch, do this on each affine patch, such that the lines join up. Now, observe where the support of the length 2 subscheme may lie in this new picture and take all stable configurations. An examination of the local equations of the blow-ups shows that we have thus found all specialisations in base codimension 2 of our length 2 subscheme 0-dimensional subscheme. These define the intersections between 4-dimensional components of $Y_0$; two 4-dimensional components described by two different configurations of points in $X_0$ intersect in the 3-dimensional component described by their common specialisation.

   This reasoning can be iterated in the following way. Assume $X_0'$ is a modification of $X_0$ in base codimension $k$, i.e.\ we have blown up $Y_1$ and $Y_2$ $k$ times. Place a configuration of two points in $X_0'$ or a single double point. Now, as before, we place a node in either the interior of $Y_1$, $Y_2$ or a $\Delta_1=\Delta_2$ component. We again draw three half lines from this node with the same slopes as before. We observe where this configuration of points may land in this new picture and retain only the stable configurations. We can again see through studying the blow-up equations that this yields all possible specialisations of our configuration.

   We have described the situation on an affine patch, but this can be extended to the projective fourfold $Y_0$ by working on each affine patch and gluing together as before.
\end{proof}

\begin{remark}
    In \cite[Theorem 1.7.1]{brown2019essential} the authors prove that the dual complex of a type III, semistable degeneration of $\mathrm{Hilb}^m(\mathrm{K3)}$ is homeomorphic to $\mathbb{C}\PP^m$ and so the computations below should give triangulations of $\mathbb{C}\PP^2$.
\end{remark}

\subsection{Quartic degeneration}\label{quartic dual comp}
Let $\Pi_Q$ denote the dual complex of $Q_0$, the special fibre of $Q\to \AAA^1$.
\begin{theorem}\label{count quartic}
    The complex $\Pi_Q$ is made up of 10 vertices, 45 edges,  110 triangles, 120 tetrahedrons and 48 4-simplices.
\end{theorem}
\begin{proof}

    The 1-simplices correspond to length 2 0-dimensional subschemes whose support falls into the $(Y_i\cap Y_j)^\circ$ locus with exactly one multiplicity, i.e.\ only one point falls into this locus or two points at the same speed (in the sense of Definition \ref{speed}). The stable subschemes of this type lie in fibres of $\mathfrak{X}\to \mathfrak{C}$ of base codimension 2. In the complement of one of the $Y_i$ components (which happens 4 times), we have the stable configurations pictured as pairs of coloured points in Figure \ref{1-simplex corner}. We have $3\times 4$ in each of the two pictures, amounting to 24. Then we add pairs of points that occur in the complement of two $Y_i$ and therefore should not be counted 4 times like the previous points. These are shown in Figure \ref{1-simplex edges}. The first picture shows both points falling into the same edge. There are 6 such edges to the tetrahedron so 6 such configurations. The second picture shows a point in a $Y_i^\circ$ and a point falling into an intersection of $Y_i$ with another component. Each $Y_i$ has 3 such intersection loci and there are 4 $Y_i$. This amounts to $4\times 3 = 12$ configurations. Finally, we must count pairs of points where both have fallen into opposite intersection edges, that is, $Y_1\cap Y_2$ with $Y_3\cap Y_4$, $Y_1\cap Y_3$ with $Y_2\cap Y_4$ and $Y_1\cap Y_4$ and $Y_2\cap Y_3$. All together, this gives
    \[
    24 + 6 + 12 + 3 = 45
    \]
    1-simplices.

    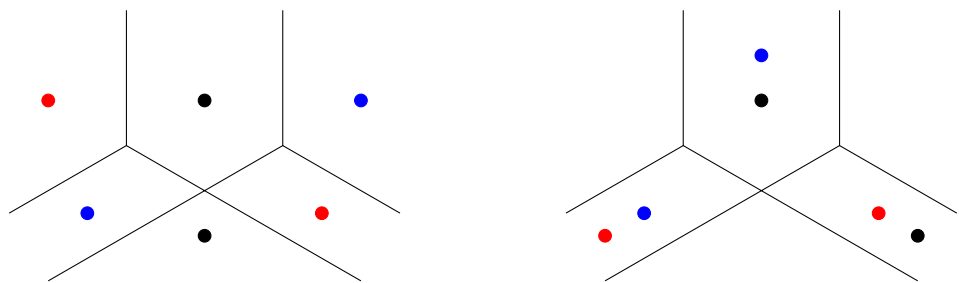
\begin{figure} 
    \begin{center}   
    \begin{tikzpicture}[scale=1.2]
        \draw   (-0.866,0.5) -- (0,0)
        (0,0) -- (0.866,0.5)       
        (-0.866,0.5) -- (-0.866,2)
        (0.866,0.5) -- (0.866,2)
        (-1.732, -1) -- (0,0)
        (-2.165,-0.249) -- (-0.866,0.5)
        
        (1.732, -1) -- (0,0)
        
        (2.165,-0.249) -- (0.866,0.5);
        %\draw (-1.516,-0.375) node[anchor=center]{$\Delta_1^{(1)}$};
        %\draw (1.516,-0.375) node[anchor=center]{$\Delta_2^{(2)}$};
        %\draw (0,1.25) node[anchor=center]{$\Delta_1^{(1)} = \Delta_2^{(2)}$};
        \filldraw[black] (0,-0.5) circle (2pt);
        \filldraw[black] (0,1) circle (2pt);
        \filldraw[blue] (-1.299,-0.25) circle (2pt);
        \filldraw[blue] (1.732,1) circle (2pt);
        \filldraw[red] (1.299,-0.25) circle (2pt);
        \filldraw[red] (-1.732,1) circle (2pt);
    \end{tikzpicture}
    \hspace{2cm}
    \begin{tikzpicture}[scale=1.2]
        \draw   (-0.866,0.5) -- (0,0)
        (0,0) -- (0.866,0.5)       
        (-0.866,0.5) -- (-0.866,2)
        (0.866,0.5) -- (0.866,2)
        (-1.732, -1) -- (0,0)
        (-2.165,-0.249) -- (-0.866,0.5)
        
        (1.732, -1) -- (0,0)
        
        (2.165,-0.249) -- (0.866,0.5);
        %\draw (-1.516,-0.375) node[anchor=center]{$\Delta_1^{(1)}$};
        %\draw (1.516,-0.375) node[anchor=center]{$\Delta_2^{(2)}$};
        %\draw (0,1.25) node[anchor=center]{$\Delta_1^{(1)} = \Delta_2^{(2)}$};
        
        \filldraw[black] (0,1) circle (2pt);
        \filldraw[blue] (0,1.5) circle (2pt);
        \filldraw[blue] (-1.299,-0.25) circle (2pt);
        \filldraw[red] (-1.732,-0.5) circle (2pt);
        \filldraw[red] (1.299,-0.25) circle (2pt);
        \filldraw[black] (1.732,-0.5) circle (2pt);
        
    \end{tikzpicture}
    \end{center}
    \caption{Stable pairs in the complement of a $Y_i$, in base codimension 2.}
    \label{1-simplex corner}
\end{figure}
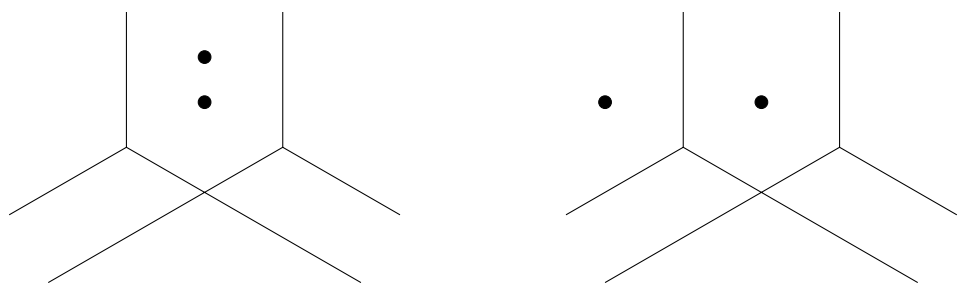
\begin{figure} 
    \begin{center}   
    \begin{tikzpicture}[scale=1.2]
        \draw   (-0.866,0.5) -- (0,0)
        (0,0) -- (0.866,0.5)       
        (-0.866,0.5) -- (-0.866,2)
        (0.866,0.5) -- (0.866,2)
        (-1.732, -1) -- (0,0)
        (-2.165,-0.249) -- (-0.866,0.5)
        
        (1.732, -1) -- (0,0)
        
        (2.165,-0.249) -- (0.866,0.5);
        %\draw (-1.516,-0.375) node[anchor=center]{$\Delta_1^{(1)}$};
        %\draw (1.516,-0.375) node[anchor=center]{$\Delta_2^{(2)}$};
        %\draw (0,1.25) node[anchor=center]{$\Delta_1^{(1)} = \Delta_2^{(2)}$};
        
        \filldraw[black] (0,1) circle (2pt);
        \filldraw[black] (0,1.5) circle (2pt);
        
    \end{tikzpicture}
    \hspace{2cm}
    \begin{tikzpicture}[scale=1.2]
        \draw   (-0.866,0.5) -- (0,0)
        (0,0) -- (0.866,0.5)       
        (-0.866,0.5) -- (-0.866,2)
        (0.866,0.5) -- (0.866,2)
        (-1.732, -1) -- (0,0)
        (-2.165,-0.249) -- (-0.866,0.5)
        
        (1.732, -1) -- (0,0)
        
        (2.165,-0.249) -- (0.866,0.5);
        %\draw (-1.516,-0.375) node[anchor=center]{$\Delta_1^{(1)}$};
        %\draw (1.516,-0.375) node[anchor=center]{$\Delta_2^{(2)}$};
        %\draw (0,1.25) node[anchor=center]{$\Delta_1^{(1)} = \Delta_2^{(2)}$};
        
        \filldraw[black] (0,1) circle (2pt);
        
        \filldraw[black] (-1.732,1) circle (2pt);
    \end{tikzpicture}
    
    \end{center}
    \caption{Stable pairs in the complement of two $Y_i$s, in base codimension 2.}
    \label{1-simplex edges}
\end{figure}

The 2-simplices correspond to stable length 2 0-dimensional subschemes of base codimension 3. Firstly, both points of the support could land in the triple intersections of $X_0$ at the same speed. We count ${4\choose 2} = 6$ for points landing in different triple intersections, and 4 for points landing in the same intersection, adding up to 10. Secondly, one of the two points of the support could land in the triple intersection and the other in a $Y_i^\circ$. In this case, there are 4 choices of triple intersection and 4 choices of $Y_i$ which makes $4\times 4 = 16$. Next, one point could fall into the triple intersection locus and one into an edge. The pair will be LW stable regardless of into which $\Delta$-component the second point falls. We therefore have 4 choices of triple intersection $\PP^1\times\PP^1$ box, and 12 choices of $\PP^1$ bundles over a $Y_i\cap Y_j$ edge, which makes $4\times 12 = 48$. Finally, both points could have hit the codimension 1 intersection locus of the degenerate K3 at different speeds. If both points hit different edges at different speeds, we have $2\times {6\choose 2} = 30$ choices. The factor of 2 comes from the fact that the order matters if we consider different edges. Both points could also hit the same edge at different speeds, in which case the order is not counted, i.e.\ there are 6 such options. All together, that makes
\[
10 + 16 +48 + 30+6 =110
\]
2-simplices.

The 3-simplices correspond to stable length 2 0-dimensional subschemes with base codimension 4. In this base codimension, the fibres of the expanded degeneration have 3 $\PP^1$ bundles over each $Y_i\cap Y_j$ and 3 $\PP^1\times\PP^1$ boxes in each triple intersection. There are ${12\choose 2} - {4\choose 2}\times 3 = 48$ stable pairs where both points fall into a triple intersection. We start by counting all pairs of 2 boxes out of the 12 available, then subtract pairs where both boxes have the same $\Delta$-type. The only other stable option here is when one point falls into a triple intersection and one into an edge. Once the $\Delta$-type of the box is determined, the $\Delta$-type of the edge is completely determined. We therefore have $6\times 4 \times 3 = 72$ such configurations, where 6 is the choice of edge, 4 is the choice of corner and 3 is the choice of box in that corner. Together, this sums to
\[
48 + 72 = 120
\]
3-simplices.

The 4-simplices of $\Pi_Q$ correspond to length 2 0-dimensional subschemes whose points of the support have fallen into the triple intersection points of the degenerate K3 at different speeds. These could be the same triple intersection or two different ones. There are 4 such triple intersections and thus 10 pairs of them, including pairs made up of twice the same one. For each of the 4 pairs of subschemes which fall into a single triple intersection, there are 3 ways for the families of length 2 0-dimensional subschemes to take stable limits, as represented in Figure \ref{3 deepest strata conf}. The 2 points of a same colour fall on different triple intersections for the other 6 pairs. Note that in this case, the order of which intersection is hit first matters. We therefore have $3\times 2$ types of limit for these pairs. Together, that makes a total of
    \[
    4\times3 + 6\times3\times2 = 48
    \]
    4-simplices.
    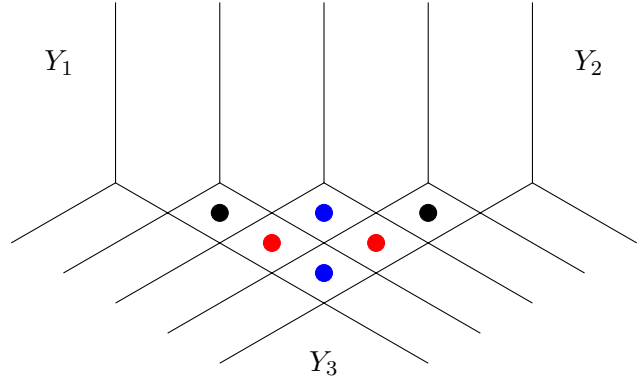
\begin{figure}
        \centering
        \begin{tikzpicture}[scale=1.6]
        \draw    (0.866,-0.5) -- (0,-1)
        (0,-1) -- (-0.866,-0.5)
        (0,0) -- (0,1.5)
        (0,0) -- (-0.866,-0.5)
        (0,0) -- (0.866,-0.5)
        (-0.866,-0.5) -- (-1.732,0)       
        (-1.732,0) -- (-1.732,1.5)
        (-1.732, -1) -- (-0.866,-0.5)
        (1.732, -1) -- (0.866,-0.5)
        (-0.866,-1.5) -- (0,-1)
        (0.866,-1.5) -- (0,-1)
        (-2.599, -0.5) -- (-1.732,0)
        (1.732, 0) -- (1.732,1.5)
        (1.732, 0) -- (0.866,-0.5)
        (1.732, 0) -- (2.599,-0.5)
        (-0.866,0) -- (-0.866,1.5)
        (-0.866,0) -- (-2.165,-0.75)
        (-0.866,0) -- (1.299,-1.25)
        (0.866,0) -- (0.866,1.5)
        (0.866,0) -- (2.165,-0.75)
        (0.866,0) -- (-1.299,-1.25)
        ;
        \draw (-2.2,1) node[anchor=center]{$Y_1$};
        \draw (2.2,1) node[anchor=center]{$Y_2$};
        \draw (0,-1.5) node[anchor=center]{$Y_3$};
        \filldraw[black] (-0.866,-0.25) circle (2pt);
        \filldraw[black] (0.866,-0.25) circle (2pt);
        \filldraw[red] (-0.433,-0.5) circle (2pt);
        \filldraw[red] (0.433,-0.5) circle (2pt);
        \filldraw[blue] (0,-0.25) circle (2pt);
        \filldraw[blue] (0,-0.75) circle (2pt);
    \end{tikzpicture}
        \caption{Deepest strata of a limit.}
        \label{3 deepest strata conf}
    \end{figure}
\end{proof}

\begin{remark}
    The numbers computed in Theorem \ref{count quartic} are exactly the same as in the simplicial complex computed by Bagchi and Datta \cite{BD}. This is extremely interesting, as we had no reason to expect such a result. Indeed, for a sphere it is known that the number of vertices in a triangulation determines the number of edges. This is not the case in higher dimension. In general knowing the number of vertices of a complex of simplices homeomorphic to $\CC\PP^n$ does not determine the number of $k$-simplices of this complex. It would be interesting to understand why this occurs in our case. Our hypothesis is that the expanded degeneration put forward here is minimal in the sense that it gives rise to the smallest possible stack of expansions which solves this problem. Note that in \cite{BD}, they study the symmetric product and their complex is unique as the simplices are uniquely determined by their vertices. In our case, however, we can find several 4-simplices which share all the same vertices.
\end{remark}

\subsection{Cube degeneration}\label{cube dual comp}
Let $\Pi_C$ denote the dual complex of $C_0$, the special fibre of $C\to \AAA^1$.
\begin{theorem}\label{count cube}
    The complex $\Pi_C$ is made up of 21 vertices, 120 edges, 420 triangles, 480 tetrahedrons and 192 4-simplices.
\end{theorem}
\begin{proof}
    This can be shown similarly to the proof of Theorem \ref{count quartic}.
\end{proof}

\bibliography{bibliography}
\bibliographystyle{alpha}

\footnotesize

\end{document}